\documentclass[a4paper,11pt,reqno,final]{amsart}

\usepackage{amsmath,amssymb,amsthm}

\numberwithin{equation}{section}
\allowdisplaybreaks

\theoremstyle{definition}
 \newtheorem{thm}{Theorem}[section]
 \newtheorem{prp}[thm]{Proposition}
 \newtheorem{lem}[thm]{Lemma}
 
 \newtheorem{dfn}[thm]{Definition}
 \newtheorem{fct}[thm]{Fact}
 \newtheorem{rmk}[thm]{Remark}
 \newtheorem*{ack}{Acknowledgements}

\newcommand{\ep}{\epsilon}
\newcommand{\ve}{\varepsilon}
\newcommand{\bbC}{\mathbb{C}}
\newcommand{\bbF}{\mathbb{F}}
\newcommand{\bbN}{\mathbb{N}}
\newcommand{\bbP}{\mathbb{P}}
\newcommand{\bbQ}{\mathbb{Q}}
\newcommand{\bbZ}{\mathbb{Z}}
\newcommand{\calA}{\mathcal{A}}
\newcommand{\calF}{\mathcal{F}}

\newcommand{\calU}{\mathcal{U}}
\newcommand{\calW}{\mathcal{W}}
\newcommand{\frakh}{\mathfrak{h}}
\newcommand{\frakS}{\mathfrak{S}}
\newcommand{\Liesl}{\mathfrak{sl}}

\newcommand{\seteq}{\mathbin{:=}}
\newcommand{\ad}{\operatorname{ad}}

\newcommand{\Ind}{\operatorname{Ind}}
\newcommand{\Sym}{\operatorname{Sym}}

\newcommand{\SSTb}[1]{\operatorname{SSTb}(#1)}
\newcommand{\Tb}[1]{\operatorname{Tb}(#1)}
\newcommand{\Tbr}[1]{\operatorname{RTb}(#1)}
\newcommand{\qtpr}[1]{\left\langle #1\right\rangle_{q,t}}

\title{Kernel function and quantum algebras}
\date{Feburary 12, 2010}
\author{B.~Feigin, A.~Hoshino, J.~Shibahara, J.~Shiraishi and S.~Yanagida}
\address{BF: Landau Institute for Theoretical Physics,
Russia, Chernogolovka, 142432, prosp. Akademika Semenova, 1a,   \newline
Higher School of Economics, Russia, Moscow, 101000,  Myasnitskaya ul., 20 and
\newline
Independent University of Moscow, Russia, Moscow, 119002,
Bol'shoi Vlas'evski per., 11}
\email{bfeigin@gmail.com}
\address{AH. Department of Mathematics, Sophia University, Kioicyo, Tokyo, 102-8554, Japan}
\email{ayumu-h@mm.sophia.ac.jp}
\address{JS, JS: Graduate School of Mathematical Sciences, University of Tokyo, Komaba, Tokyo
153-8914, Japan}
\email{shibahara@ms.u-tokyo.ac.jp}
\email{shiraish@ms.u-tokyo.ac.jp}
\address{SY: Kobe University, Department of Mathematics, Rokko, Kobe 657-8501, Japan}
\email{yanagida@math.kobe-u.ac.jp}

\begin{document}

\begin{abstract}
We introduce an analogue $K_n(x,z;q,t)$ of the 
Cauchy-type kernel function for the Macdonald polynomials,
being constructed 
in the tensor product of the ring $\Lambda_\bbF$ of symmetric functions and the 
commutative algebra $\calA$ over the degenerate $\bbC \bbP^1$. 
We show that a certain restriction of $K_n(x,z;q,t)$  with respect to the variable $z$ 
is neatly described by the tableau sum formula of Macdonald polynomials. 
Next, we demonstrate that the level $m$ representation of 
the Ding-Iohara quantum algebra $\calU(q,t)$ 
naturally produces the currents of the deformed $\calW_{q,p}(\Liesl_n)$. 
Then we remark that the $K_n(x,z;q,t)$ 
emerges in the highest-to-highest correlation function of the deformed $\calW_{q,p}(\Liesl_n)$ algebra.
\end{abstract}

\maketitle


\section{Kernel function}

\subsection{The algebra $\calA$}
We briefly recall the definition and the basic facts 
about the commutative algebra $\calA$ introduced in \cite{FHHSY:2009}.
Let $q_1,q_2$ be two independent indeterminates 
and set $q_3\seteq 1/q_1q_2$. 
We also use the symbols $\bbF\seteq \bbQ(q_1,q_2)$, 
$\bbN\seteq\{0,1,2,\ldots\}$ and $\bbN_+\seteq\{1,2,\ldots\}$.

For $n,k\in\bbN_+$, we define two operators $\partial^{(0,k)},\partial^{(\infty,k)}$ acting on the space of symmetric rational functions in $n$ variables $x_1,\ldots,x_n$ by
\begin{align*}
\begin{array}{l l l l l l l l l}
\partial^{(0,k)}&: &f &\mapsto 
 &\displaystyle 
   \dfrac{n!}{(n-k)!} \lim_{\xi \to 0} 
   f(x_1,\ldots,x_{n-k},\xi x_{n-k+1},\xi x_{n-k+2},\ldots,\xi x_n)
\\
\partial^{(\infty,k)}&: &f &\mapsto 
 &\displaystyle 
    \dfrac{n!}{(n-k)!} \lim_{\xi \to \infty}
    f(x_1,\ldots,x_{n-k},\xi x_{n-k+1},\xi x_{n-k+2},\ldots,\xi x_n)
\end{array}
\end{align*}
whenever the limit exists. We also set $\partial^{(0,k)} c=0, \partial^{(\infty,k)} c=0$ for $c\in \bbF$. Finally we define $\partial^{(0,0)}$ and $\partial^{(\infty,0)}$ to be the identity operator.

\begin{dfn}
For $n\in\bbN$, the vector space $\calA_n=\calA_n(q_1,q_2,q_3)$ is defined by the following conditions (i), (ii), (iii) and (iv).
\\
(i)
$\calA_0 \seteq \bbF$. For $n\in\bbN_+$, $f(x_1,\ldots,x_n)\in \calA_n$ is a rational function with coefficients in $\bbF$, and symmetric with respect to the $x_i$'s.
\\
(ii) For $n\in\bbN$, $0\leq k\leq n$ and $f \in \calA_n$, 
the limits $\partial^{(\infty,k)}f$ and $\partial^{(0,k)}f$ both exist and coincide: $\partial^{(\infty,k)}f=\partial^{(0,k)}f$ (degenerate $\bbC \bbP^1$ condition).
\\
(iii)
The poles of $f\in \calA_n$ are located only on the diagonal $\{(x_1,\ldots,x_n) \mid \exists (i,j), i\neq j ,x_i=x_j\}$, and the orders of the poles are at most two.
\\
(iv) For $n\geq 3$, $f\in \calA_n$ satisfies the wheel conditions
\begin{align*}
f( x_1,q_1 x_1,q_1 q_2 x_1,x_4,\ldots)=0,\qquad
f( x_1,q_2 x_1,q_1 q_2 x_1,x_4,\ldots)=0.
\end{align*}

Then we set the graded vector space 
$\calA=\calA(q_1,q_2,q_3)\seteq\bigoplus_{n\geq 0}\calA_n$.
\end{dfn}


\begin{dfn} 
For an $m$-variable symmetric rational function $f$ and an $n$-variable symmetric 
rational function $g$, we define an $(m+n)$-variable symmetric rational function $f*g$ by 
\begin{align}\label{eq:star}
 (f*g)(x_1,\ldots,x_{m+n})&\seteq
  \Sym
  \bigg[
   f(x_1,\ldots,x_m) g(x_{m+1},\ldots,x_{m+n})
   \prod_{\substack{1\le\alpha\le m\\m+1\le\beta\le m+n}}
   \omega(x_\alpha,x_\beta)
  \bigg].
\end{align}
Here $\omega(x,y)$ is the rational function
\begin{align}\label{eq:omega}
\omega(x,y)=\omega(x,y;q_1,q_2,q_3)
\seteq\dfrac{(x-q_1 y)(x - q_2 y)(x-q_3 y)}{(x-y)^3},
\end{align}
and the symbol $\Sym$ means
$\Sym (f(x_1,\ldots,x_n))\seteq
(1/n!) \;\sum_{\sigma\in\frakS_{n}}
f(x_{\sigma(1)},\ldots,x_{\sigma(n)})$.
\end{dfn}

\begin{fct}[{\cite[Theorem 1.5]{FHHSY:2009}}]\label{thm:1}
$\calA$ is closed with respect to $*$, and the pair $(\calA,*)$ is a unital associative commutative algebra. The Poincar\'e series is $\sum_{n\ge 0} (\dim_\bbF \calA_n) z^n=\prod_{m\ge 1}(1-z^m)^{-1}$.
\end{fct}

\subsection{The ring $\Lambda_{\bbF}$ of symmetric functions}
As for the notations and definitions concerning the partitions, 
we basically follow the notation in \cite{M:1995:book}. A partition of $n\in \bbN$ is a sequence  $\lambda=(\lambda_1,\lambda_2,\ldots)$ of non-negative integers satisfying $\lambda_1\geq\lambda_2\geq\cdots$. We define $|\lambda|\seteq\lambda_1+\lambda_2+\cdots$, $\ell(\lambda)\seteq\#\{i\mid\lambda_i\neq0\}$, and write $\lambda \vdash n$ if $|\lambda|=n$. We denote the conjugate (transpose) of a partition $\lambda$ by $\lambda'$. We work with the dominance partial ordering defined as :
$\lambda\geq \mu \overset{\rm def}{\iff}
|\lambda|=|\mu|,\  \lambda_1+\cdots+\lambda_i\geq \mu_1+\cdots+\mu_i
\mbox{ for all } i\geq 1$.

We recall some basic facts about the ring of symmetric functions. 
As was in \cite{FHHSY:2009}, we 
set $q_1=q^{-1},q_2=t$ (hence $q_3=q t^{-1}$) and
$\bbF=\bbQ(q_1,q_2)=\bbQ(q,t)$.
Let $\Lambda_\bbF$ be the ring of symmetric functions 
over the base field $\bbF$, constructed in the category of graded 
ring with the projection operators $\rho_{m,n}:f(x_1,\ldots,x_m)\mapsto
f(x_1,\ldots,x_n,0\ldots,0)$.

Let $p_n(x)\seteq \sum_i x_i^n$ be the power sum function. 
For a partition $\lambda=(\lambda_1,\lambda_2,\ldots)$,
the monomial symmetric function 
is defined by 
$m_\lambda(x)\seteq \sum_{\alpha}x^\alpha$,
where $\alpha$ runs over all the distinct permutations of $\lambda$.
The elementary symmetric function $e_n(x)$ is defined by
the generating function
$E(y)\seteq\prod_{i}(1+x_i y)=\sum_{n\geq 0} e_n(x) y^n$. 
Set 
$G(y)\seteq\prod_i \{(t x_i y;q)_\infty / (x_i y;q)_\infty\}
=\sum_{n\geq 0} g_n(x;q,t) y^n$,
 where $(x;q)_\infty\seteq\prod_{i\geq 0}(1-q^i x)$.
For a partition
$\lambda=(\lambda_1,\lambda_2,\ldots)$ set 
$p_\lambda\seteq p_{\lambda_1}p_{\lambda_2}\cdots$. 
Similarly we write
$e_\lambda\seteq e_{\lambda_1}e_{\lambda_2}\cdots$ and
$g_\lambda\seteq g_{\lambda_1}g_{\lambda_2}\cdots$.
It is known that $\{p_\lambda\}$, $\{m_\lambda\}$,   
$\{e_\lambda\}$ and $\{g_\lambda\}$ form bases of $\Lambda_\bbF$.

Recall Macdonald's scalar product
$\qtpr{p_\lambda,p_\mu}
\seteq
\delta_{\lambda,\mu}
\prod_{i \ge 1}i^{m_i}m_i!
\prod_{j \ge 1}(1-q^{\lambda_j})/(1-t^{\lambda_j})$,
where $m_i$ denotes the number of parts $i$ in the partition $\lambda$. 
For any dual bases $\{u_\lambda\}$ and $\{v_\lambda\}$, we have
\begin{align}\label{eq:kernel:original}
\Pi(x,y;q,t)\seteq \prod_{i,j} 
\dfrac{(t  x_iy_j;q)_\infty}{(x_i y_j;q)_\infty}
=\sum_\lambda u_\lambda(x)v_\lambda(y).
\end{align}

It is known that $\{m_\lambda\}$ and $\{g_\lambda\}$ form  
dual  bases, namely we have $\qtpr{m_\lambda,g_\mu}=\delta_{\lambda,\mu}$.

Macdonald polynomials $P_\lambda(x;q,t)$ are uniquely characterized by 
(a) the triangular expansion 
$P_\lambda=m_\lambda+\sum_{\mu<\lambda} a_{\lambda\mu} m_\mu$
($a_{\lambda\mu}\in \bbF$), and (b) the 
orthogonality $\qtpr{P_\lambda,P_\mu}=0$ if $\lambda\neq \mu$.

Se set 
\begin{align}
\label{eq:b_def}
b_\lambda(q,t)\seteq\qtpr{P_\lambda(z;q,t),P_\lambda(z;q,t)}^{-1},
\quad
Q_\lambda(z;q,t)\seteq b_\lambda(q,t)P_\lambda(z;q,t).
\end{align}
Then $\{Q_\lambda\}$ forms a dual basis to $\{P_\lambda\}$.

\subsection{The isomorphism $\iota:\Lambda_\bbF \rightarrow \calA$}
Both $\Lambda_\bbF$ and $\calA$ are commutative rings 
having the same 
Poincar\'e series $
\sum_{n\ge 0} (\dim_\bbF \Lambda_\bbF^n) z^n=
\sum_{n\ge 0} (\dim_\bbF \calA_n) z^n=\prod_{m\ge 1}(1-z^m)^{-1}$,
where $\Lambda_\bbF^n$ denotes the ring of symmetric functions of degree $n$.
Moreover it was shown in \cite{FHHSY:2009} 
that there is a natural way to identify 
$\Lambda_\bbF$ and $\calA$ from the point of view of the 
free field construction of the Macdonald operators. 
Based on the finding in \cite{FHHSY:2009} 
we give an isomorphism $\iota:\Lambda_\bbF \rightarrow \calA$
as follows.

For $p \in \bbF$, let
\begin{align}
&\ep_n(z_1,z_2,\ldots,z_n;p)\seteq
\prod_{1\le i<j\le n}\dfrac{(z_i-p z_j)(z_i-p^{-1}z_j)}{(z_i-z_j)^2},
\label{eq:ep}
\end{align}
and set $\ep_{\lambda}(z;p):= (\ep_{\lambda_1}*\ep_{\lambda_2}* \cdots*\ep_{\lambda_l})
(z;p)$ for a multi-index 
$\lambda = (\lambda_1, \lambda_2, \ldots, \lambda_l)$. 

\begin{fct}[{\cite[Propositions 2.20 \& 2.23]{FHHSY:2009}}]
For $i=1,2,3$, 
$\{ \ep_{\lambda}(z;q_i) \}_{\lambda\vdash n}$ forms a basis of $\calA_n$.
\end{fct}

Let us write the expansions of $P_\lambda$ in the bases $\{e_\mu\}$ 
and$\{g_\mu\}$ by
\begin{align}
&\label{eq:gtoP}
P_\lambda(z;q,t)=\sum_{\mu\ge\lambda'}c_{\lambda\mu}^{e\to P}(q,t)e_\mu(z;q,t),
\quad
P_\lambda(x;q,t)=\sum_{\mu\ge\lambda}c_{\lambda\mu}^{g\to P}(q,t)g_\mu(x;q,t).
\end{align}
A detailed study of the algebra $\calA$ with the 
help of the free field representation allowed us to establish
the following equality. 

\begin{fct}[{\cite[\S 3 E]{FHHSY:2009}}]
Set the next two elements in $\calA$.
\begin{align}
&f^{(q^{-1})}_\lambda(z;q,t)
\seteq
\dfrac{t^{-|\lambda|}}{(1-t^{-1})^{|\lambda|}|\lambda|!}
\sum_{\mu\ge\lambda'}c_{\lambda\mu}^{e\to P}(q,t)\ep_{\mu}(z;q)
\dfrac{|\mu|!}{\prod_{i=1}^{\ell(\mu)}\mu_i!},
\label{eq:fq}
\\
&f^{(t)}_\lambda(z;q,t)
\seteq
\dfrac{(-1)^{|\lambda|}}{(1-q)^{|\lambda|}|\lambda|!}
\sum_{\mu\ge\lambda}c_{\lambda\mu}^{g\to P}(q,t)\ep_{\mu}(z;t)
\dfrac{|\mu|!}{\prod_{i=1}^{\ell(\mu)}\mu_i!}.
\label{eq:ft}
\end{align}
Then we have $f^{(q^{-1})}_\lambda(z;q,t)=f^{(t)}_\lambda(z;q,t)$\footnote{
Note that 
the first and second lines of Page 25 in \cite{FHHSY:2009} 
contains typos and
should be read as \eqref{eq:fq} and \eqref{eq:ft}.}.
\end{fct}

\begin{dfn}
Let 
$F_\lambda(z;q,t)\seteq f^{(q^{-1})}_\lambda(z;q,t)=f^{(t)}_\lambda(z;q,t)$.
\end{dfn}

\begin{dfn}
Define the isomorphism $\iota : \Lambda_\bbF \rightarrow \calA$ by 
\begin{align*}
\iota(e_\lambda)=
\dfrac{t^{-|\lambda|}}{(1-t^{-1})^{|\lambda|}}
\dfrac{1}{\prod_{i=1}^{\ell(\mu)}\lambda_i!}\ep_{\lambda}(z;q).
\end{align*}
\end{dfn}

\begin{prp}\label{prp:iota}
(1)
We have
\begin{align*}
\iota(g_\lambda)=
\dfrac{(-1)^{-|\lambda|}}{(1-q)^{|\lambda|}}
\dfrac{1}{\prod_{i=1}^{\ell(\mu)}\lambda_i!}\ep_{\lambda}(z;t).
\end{align*}

(2)
We have $\iota(P_\lambda)=F_\lambda(z;q,t)$.
\end{prp}

\begin{proof}
(1) By the Wronski relation given in \cite[Proposition 3.11]{FHHSY:2009}.

(2) By \eqref{eq:gtoP} and the definitions of $\iota$ and $F_\lambda$.
\end{proof}

\begin{rmk}
To explain the importance of the element
$F_\lambda(z;q,t)$,
we recall the Gordon filtration on $\calA$.
For $p\in \bbF$ and $\lambda=(\lambda_1,\ldots,\lambda_l)\vdash n$,
we defined a linear map
\begin{align}\label{eq:special}
\begin{array}{l c c l}
\varphi_\lambda^{(p)} : 
&\calA_n          &\longrightarrow&\bbF(y_1,\ldots,y_l)\\
&f(z_1,\ldots,z_n)&\mapsto &f(y_1,p y_1,\ldots,p^{\lambda_1-1}y_1,\\
&                 &        &\phantom{f(} y_2,p y_2,\ldots,p^{\lambda_2-1}y_2,\\
&                 &        &\phantom{f(} \ldots \\
&                 &        &\phantom{f(} y_l,p y_l\ldots,p^{\lambda_l-1}y_l),
\end{array}
\end{align}
called the \emph{specialization map}.
The Gordon filtration is given by
$\calA_{n,\lambda}^{(q_i)}
\seteq\bigcap_{\mu\not\le\lambda}\ker\varphi_\mu^{(q_i)}$ for $i=1,2,3$.
Then by \cite[Theorem 1.19]{FHHSY:2009} , 
$\calA_{n,\mu}^{(q^{-1})} \cap \calA_{n,{\mu'}}^{(t)}$ is one dimensional 
and is spanned by $F_\lambda(z;q,t)$.
\end{rmk}

\subsection{The kernel function}
Now we are ready to study the kernel function from the point of view of the algebra $\calA$.
\begin{dfn}
Introduce $K_n(x,z;q,t)\in \Lambda_\bbF^n \otimes \calA_n$ 
as
\begin{align*}
K_n(x,z;q,t) \seteq \sum_{\lambda\vdash n} Q_\lambda(x)F_\lambda(z;q,t).
\end{align*}
\end{dfn}

\begin{rmk}
The name ``kernel function'' comes from $\Pi(x,y)$ 
in \eqref{eq:kernel:original}.
By Proposition \ref{prp:iota} (2),
we have, in a suitable completion of $\Lambda_\bbF\otimes \calA$,
\begin{align*}
\sum_{n\geq 0} K_n(x,z;q,t)=\sum_\lambda Q_\lambda(x) \iota(P_\lambda(y)),
\end{align*}
where $\lambda$ runs over all the partitions of every non-negative integer.
Thus $K_n$ is a homogeneous component of the analogue of $\Pi(x,y)$.
\end{rmk}

\begin{prp}\label{prp:K2}
In 
$\Lambda_\bbF\otimes \calA$ 
we have
\begin{align}\label{eq:K2}
K_n(x,z;q,t)=\dfrac{(-1)^n}{(1-q)^n n!}
\sum_{\lambda\vdash n}m_\lambda(x)
\epsilon_\lambda(z;t) 
\dfrac{|\lambda|!}{\prod_{i=1}^{\ell(\lambda)}\lambda_i!}
\end{align}
\end{prp}

\begin{proof}
First we show 
\begin{align}\label{eq:Qtom}
m_\lambda(x)=\sum_{\mu\le\lambda}c_{\mu\lambda}^{g\to P}(q,t)Q_\mu(x;q,t).
\end{align}
Since $\{Q_\mu(x;q,t)\}$ is a basis of $\Lambda_\bbF$, 
we can expand $m_\lambda(x)=\sum_{\nu}c_{\nu\lambda} Q_\nu(x;q,t)$ 
with $c_{\nu\lambda}\in\bbF$.
Then the pairing $\qtpr{m_\lambda,P_\mu}$ is calculated as
$\qtpr{m_\lambda,P_\mu}
=\qtpr{\sum_{\nu}c_{\nu\lambda} Q_\nu(z;q,t),P_\mu}
=c_{\mu\lambda}$,
where 
we used the fact that $\{P_\lambda\}$ 
and $\{Q_\lambda\}$ are dual.
On the other hand, by \eqref{eq:gtoP}, we have 
$\qtpr{m_\lambda,P_\mu}
=\qtpr{m_\lambda,\sum_{\nu\ge\mu}c_{\mu\nu}^{g\to P}(q,t)g_\mu}
=c_{\mu\lambda}^{g\to P}(q,t)$.
Comparing both expressions, we obtain \eqref{eq:Qtom}.

Then we have
\begin{align*}
\text{RHS of } \eqref{eq:K2}
&=\dfrac{(-1)^n}{(1-q)^n n!}\sum_{\lambda\vdash n}
  \sum_{\mu\le\lambda}c_{\mu\lambda}^{g\to P}(q,t)Q_\mu(x;q,t)
  \ep_\lambda(z;t)
  \dfrac{|\lambda|!}{\prod_{i=1}^{\ell(\lambda)}\lambda_i!}\\
&=\dfrac{(-1)^n}{(1-q)^n n!}\sum_{\mu\vdash n}Q_\mu(x;q,t)
  \sum_{\lambda\ge\mu}c_{\mu\lambda}^{g\to P}(q,t)
  \ep_\lambda(z;t)
  \dfrac{|\lambda|!}{\prod_{i=1}^{\ell(\lambda)}\lambda_i!}
=\sum_{\mu\vdash n}Q_\mu(x;q,t)F_\mu(z;q,t).
\end{align*}
\end{proof}

Consider the case of finitely many variables and set $x=(x_1,x_2,\ldots,x_m)$.
Also let $z=(z_1,z_2,\ldots,z_n)$ 
be the set of variables for the elements in $\calA_n$.

\begin{prp}\label{prp:K}
We have
\begin{align}
\label{eq:K:dfn}
K_n(x,z;q,t) =
\dfrac{(-1)^{n}}{(1-q)^{n}n!}
\sum_{i_1=1}^m \sum_{i_2=1}^m \cdots \sum_{i_n=1}^m 
 x_{i_1}x_{i_2}\cdots x_{i_n}
 \prod_{1\le\alpha<\beta\le n}\gamma_{i_\alpha,i_\beta}(z_\alpha,z_\beta;q,t),
\end{align}
where the function $\gamma_{i,j}(z,w;q,t)$ is given by
\begin{align}\label{eq:gamma}
\gamma_{i,j}(z,w;q,t)\seteq
\begin{cases}
\dfrac{(z-t w)(z-t^{-1}w)}{(z-w)^2} & i=j,\\
\dfrac{(z-q^{-1} w)(z-t w)(z-q t^{-1}w)}{(z-w)^3} & i<j,\\
\dfrac{(z-q w)(z-t^{-1} w)(z-q^{-1} t w)}{(z-w)^3}& i>j. 
\end{cases}
\end{align}
\end{prp}

\begin{proof}
Note that we have 
\begin{align}\label{eq:gamma-omega}
\gamma_{i,j}(z,w;q,t) =
\begin{cases}
\ep_2(z,w;t) & i=j,\\
\omega(z,w;q^{-1},t,q t^{-1}) & i<j,\\
\omega(z,w;q,t^{-1},q^{-1} t)=\omega(w,z;q^{-1},t,q t^{-1})& i>j, 
\end{cases}
\end{align}
which is obtained from \eqref{eq:omega}, \eqref{eq:ep} and \eqref{eq:gamma}.
Thus we have
\begin{align*}
&\sum_{i_1=1}^m \sum_{i_2=1}^m \cdots \sum_{i_n=1}^m 
 x_{i_1}x_{i_2}\cdots x_{i_n}
 \prod_{1\le\alpha<\beta\le n}\gamma_{i_\alpha,i_\beta}(z_\alpha,z_\beta;q,t)\\
 &=
\sum_{I_1,\ldots,I_m}x_1^{a_1}x_2^{a_2}\cdots x_m^{a_m}
\prod_{k=1}^m\ep_{a_k}(z_{I_k};t)
\prod_{1\le i<j\le m}\prod_{\alpha\in I_i,\beta\in I_j}
\omega(z_\alpha,z_\beta;q^{-1},t,q t^{-1}),
\end{align*}
where  $I_k$ ($k=1,2,\ldots,m$) is a subset of $\{1,2,\ldots,n\}$ 
such that $|I_k|=a_k$, $I_1\cup I_2\cup\cdots\cup I_m=\{1,\ldots,n\}$.
Using the multi-index notation $a=(a_1,\ldots,a_m)\in\bbN^m$, we have
\begin{align*}
=\sum_{a\in\bbN^m, |a|=n}x^{a} 
 \dfrac{n!}{\prod_{k=1}^m a_k!} \ep_{a}(z;t)
\end{align*}
with $|a|\seteq a_1+\cdots+a_m$.
Applying $\frakS_m$ on the running index $a$ and averaging  them, we have
\begin{align*}
=\dfrac{1}{n!}\sum_{\sigma\in\frakS_m}\sum_{a\in\bbN^m,|a|=n}
  x^{\sigma(a)}\dfrac{n!}{\prod_{k=1}^m a_k!} \ep_{\sigma(a)}(z;t).
\end{align*}
Dividing 
$\frakS_m$ by the stabilizer $\text{Stab}(a)$ 
of $a\in\bbN^m$ and using the commutativity of $\calA$, we have
\begin{align*}
=\dfrac{1}{n!}\sum_{a\in\bbN^m,|a|=n} 
 \#\text{Stab}(a) \dfrac{n!}{\prod_{k=1}^m a_k!} \Big(
 \sum_{\overline{\sigma}\in\frakS_m/\text{Stab}(a)}
 x^{\overline{\sigma}(a)} \Big)\ep_{a}(z;t).
\end{align*}
Then we obtain the result 
by taking a partition $\lambda$ as the running index.
\end{proof}

\subsection{Macdonald's tableau sum formula}\label{subsec:tableau}
We recall the tableau sum formula for the Macdonald polynomials.

Let $\Tb{\lambda;m}$ denotes the set of 
all the 
ways of drawing numbers $1,2,\ldots,m$ 
into the Young diagram of shape $\lambda$
\emph{without any conditions}.
Reading the numbers from 
left to right then 
top to bottom, namely in the English reading manner,
we get a bijection between 
$\Tb{\lambda;m}$ and  
the set   $ \{1,2,\ldots,m\}^n$.

Let 
$\Tbr{\lambda;m}$ denotes the subset of $\Tb{\lambda;m}$ 
in which the numbers in each row are arranged in non-decreasing manner.
The element of $\Tbr{\lambda;m}$ is uniquely
specified by the set of numbers $\theta_{i,j}$ which denote the number of $j$ 
in the $i$-th row.
We have $\lambda_i = \sum_{k=1}^{n} \theta_{i,k}$ for $1 \leq i \leq n$.
Next we introduce a sequence 
$\lambda^{(j)}=(\lambda^{(j)}_1,\lambda^{(j)}_2,\ldots)$
by setting $\lambda^{(j)}_i\seteq \sum_{k=1}^j\theta_{i,k}$.
It is clear that we have 
$\emptyset=\lambda^{(0)}\subset \lambda^{(1)}\subset\cdots \subset \lambda^{(m)}=\lambda$. 
Note that $\lambda^{(j)}$ may not be a partition.

Let $\SSTb{\lambda;m}$  be the set of semi-standard tableaux. 
A semi-standard tableau $T$ is 
expressed as a sequence of partitions 
$\emptyset=\lambda^{(0)}\subset \lambda^{(1)}\subset\cdots \subset \lambda^{(m)}=\lambda$,
where the skew diagrams $\lambda^{(k)}/\lambda^{(k-1)}$ ($k=1,2,\ldots,m$)
are horizontal strips. 
We have $\theta_{i,j}=0$ for $i>j$,
 $\lambda_i = \sum_{k=i}^{n} \theta_{i,k}$  for $1 \leq i \leq n$,
and
\begin{align*}
0\leq \theta_{i,j}\leq \lambda_i-\lambda_{i+1}-
\sum_{k=j+1}^{\ell(\lambda)}(\theta_{i,k}-\theta_{i+1,k})
\end{align*}
for $1\leq i<j\leq \ell(\lambda)$.

It is known that the $b_\lambda(q,t)$ in \eqref{eq:b_def} 
has the factorized form.
\begin{align}
\label{eq:b_fact}
Q_\lambda(x;q,t)=b_\lambda(q,t) P_\lambda(x;q,t),\quad
b_\lambda(q,t) =
\prod_{s\in\lambda}\dfrac{1-q^{a(s)}t^{\ell(s)+1}}{1-q^{a(s)+1}t^{\ell(s)}},
\end{align}
where for a box $s=(i,j)$ of 
$\lambda$,
$a(s)\seteq \lambda_i-j$ is the arm-length and 
$\ell(s)\seteq \lambda'_j-i$ is the leg-length.

The $P_\lambda(x;q,t)$ has the tableau sum formula:
\begin{align*}
P_\lambda(x;q,t)=\sum_{T\in\SSTb{\lambda;m}}x^T \psi_T(q,t).
\end{align*}
Here the coefficient $\psi(q,t)\in\bbF$ is determined by
\begin{align}\label{eq:tableau:formula}
\begin{split}
&\psi_T(q,t)\seteq \prod_{k=1}^m \psi_{\lambda^{(k)}/\lambda^{(k-1)}}(q,t),\\
&\psi_{\lambda/\mu}(q,t)\seteq 
 \prod_{1\le i \le j\le \ell(\mu)}^n 
 \dfrac{f(q^{\mu_i-\mu_j}t^{j-i})f(q^{\lambda_i-\lambda_{j+1}}t^{j-i})}
       {f(q^{\lambda_i-\mu_j}t^{j-i})f(q^{\mu_i-\lambda_{j+1}}t^{j-i})},\quad
 f(u)\seteq\dfrac{(t u;q)_\infty}{(q u;q)_\infty}.
\end{split}
\end{align}

The next proposition is obtained by simple combinatorics, 
and we omit the proof for lack of space.
\begin{prp}\label{prp:tableau}
Let $T\in \Tbr{\lambda;m}\setminus \SSTb{\lambda;m}$
and regard $T$ as a sequence $\lambda^{(j)}$ explained as above. 
Then  $\psi_T(q,t)$ calculated from 
(\ref{eq:tableau:formula}) vanishes.
\end{prp}

\subsection{Tableau sum formula and $K_n(x,z;q,t)$}

Now we investigate the relationship between the function 
$K_n(x,z;q,t)$ and the tableau formula of Macdonald polynomial.
We fix a natural number $m$ and 
consider the case $x=(x_1,\ldots,x_m)$.

In order to state the main result, 
we need to consider the composition of the specialization maps
$\varphi_\lambda^{(p)}$ of \eqref{eq:special}.
For a partition $\lambda=(\lambda_1,\ldots,\lambda_l)$ of $n$ 
and $\zeta\in\bbF$, 
we define the map $\widetilde{\varphi}^{(\zeta)}_\lambda$ by
\begin{align}\label{eq:zeta}
\begin{array}{l c c l}
\widetilde{\varphi}^{(\zeta)}_\lambda\seteq
\varphi_{(l)}^{(\zeta)}\circ\varphi_\lambda^{(q^{-1})} : 
&\bbF(z_1,\ldots,z_n)     &\longrightarrow&\bbF(y)\\
&f(z_1,\ldots,z_n)&\mapsto&f(y,q^{-1} y,\ldots,q^{-(\lambda_1-1)}y,\\
& & &\hskip 1em \zeta y,q^{-1}\zeta y,\ldots,q^{-(\lambda_2-1)}\zeta y,\\
& & &\hskip 1em \ldots,\\
& & &\hskip 1em \zeta^{l-1}y,q^{-1}\zeta^{l-1}y,
                \ldots,q^{-(\lambda_l-1)}\zeta^{l-1}y).
\end{array}
\end{align}
Here the map $\varphi_{(l)}^{(\zeta)}$ denotes
the substitution 
$\varphi_{(l)}^{(\zeta)}g(y_1,\ldots,y_{l})=g(y,\zeta y,\ldots,\zeta^{l-1} y)$.

\begin{thm}\label{thm:tableau}
For partitions $\mu,\lambda$ of $n$,
$\widetilde{\varphi}^{(\zeta)}_\lambda(F_\mu / F_\lambda)$ is regular at 
$\zeta=t$ and its value is $\delta_{\lambda,\mu}$.
\end{thm}

Our proof uses
the tableau sum formula of $P_\lambda(x;q,t)$.
Let us express the statement as
\begin{align*}
\lim_{\zeta\to t}\widetilde{\varphi}^{(\zeta)}_\lambda
\dfrac{F_\mu(z;q,t)}{F_\lambda(z;q,t)}
=\delta_{\lambda,\mu}.
\end{align*}
Then by using Proposition~\ref{prp:K2}, 
it can be rewritten into the next equivalent form.
\begin{align}\label{eq:thm:tableau}
\lim_{\zeta\to t}\widetilde{\varphi}^{(\zeta)}_\lambda 
\dfrac{K_n(x,z;q,t)}{F_\lambda(z;q,t)}
=Q_\lambda(x;q,t).
\end{align}

Regard $T=(i_1,i_2,\ldots,i_n)\in \{1,2,\ldots,m\}^n $
as an element of $\Tb{\lambda;m}$.
For simplicity we set 
\begin{align*}
\gamma_T(z)\seteq
\prod_{1\le\alpha<\beta\le n}\gamma_{i_\alpha,i_\beta}(z_\alpha,z_\beta;q,t).
\end{align*}
We also use the same symbol for the cases $T\in \Tbr{\lambda;m}$ and
 $T\in \SSTb{\lambda;m}$.
By Proposition \ref{prp:K2}, \eqref{eq:thm:tableau} is equivalent to 
\begin{align*}
\dfrac{(-1)^{n}}{(1-q)^{n}n!}
\sum_{T\in \Tb{\lambda;m}}
x^{T}
\lim_{\zeta\to t}\widetilde{\varphi}^{(\zeta)}_\lambda 
\dfrac{\gamma_T(z)}{F_\lambda(z;q,t)}
=Q_\lambda(x;q,t).
\end{align*}
It is easy to see from the definition of $\gamma_{i,j}$ that 
all the terms with $T\in \Tb{\lambda;m}\setminus\Tbr{\lambda;m}$
vanish after
the first 
specialization $\varphi^{(q^{-1})}_\lambda$.
Thus we may replace $\sum_{T\in \Tb{\lambda;m}}$
by $\sum_{T\in \Tbr{\lambda;m}}$.

Hence it is enough to show that for $T\in \Tbr{\lambda;m}$ we have
\begin{align*}
\dfrac{(-1)^{n}}{(1-q)^{n}n!}\lim_{\zeta\to t}\widetilde{\varphi}^{(\zeta)}_\lambda 
\dfrac{\gamma_T(z)}{F_\lambda(z;q,t)}=b_\lambda(q,t)\psi_T(q,t).
\end{align*}

We prove this in two steps.
\begin{prp}\label{prp:b-psi}
Let $D\in\SSTb{\lambda;m}$ 
given by $\theta_{i,i}=\lambda_i$ and $\theta_{i,j}=0$ for $i\neq j$.
Then we have
\begin{align}
\label{eq:b}
&\dfrac{(-1)^{n}}{(1-q)^{n}n!}
 \lim_{\zeta\to t}\widetilde{\varphi}^{(\zeta)}_\lambda 
 \dfrac{\gamma_D(z)}{F_\lambda(z;q,t)}=b_\lambda(q,t),\\
\label{eq:psi}
&\lim_{\zeta\to t}\widetilde{\varphi}^{(\zeta)}_\lambda 
 \dfrac{\gamma_T(z)}{\gamma_D(z)}=\psi_T(q,t).
\end{align}
\end{prp}
\begin{proof}
The proof is postponed until \S \ref{subsec:prf:b-psi}.
\end{proof}

\section{Ding-Iohara algebra and kernel function}

In this section all objects are defined on $\widetilde{\bbF}\seteq\bbQ(q^{1/2},t^{1/2})$. 
We will also use $p\seteq q/t$.

\subsection{Review of the Ding-Iohara algebra $\calU(q,t)$}
Recall that the Ding-Iohara algebra \cite{DI:1997} was introduced 
as a generalization of the quantum affine algebra, 
which respects the structure of the Drinfeld coproduct. 
In \cite[Appendix A]{FHHSY:2009}, 
the authors introduced a version $\calU(q,t)$ of the Ding-Iohara algebra 
having two parameters $q$ and $t$.

\begin{dfn}\label{dfn:uqt}
Set
\begin{align*}
g(z)\seteq\dfrac{G^+(z)}{G^-(z)},\qquad
G^\pm(z)\seteq(1-q^{\pm1}z)(1-t^{\mp 1}z)(1-q^{\mp1}t^{\pm 1}z).
\end{align*}
Then we define $\calU(q,t)$ to be a unital associative algebra generated 
by the Drinfeld currents 
\begin{align*}
x^\pm(z)=\sum_{n\in \bbZ}x^\pm_n z^{-n},\qquad 
\psi^\pm(z)=\sum_{\pm n\in \bbN}\psi^\pm_n z^{-n},
\end{align*}
and the central element $\gamma^{\pm 1/2}$, satisfying the defining relations
\begin{align*}
\begin{array}{ll}
 \psi^\pm(z) \psi^\pm(w)= \psi^\pm(w) \psi^\pm(z),
&\psi^+(z)\psi^-(w)=
 \dfrac{g(\gamma^{+1} w/z)}{g(\gamma^{-1}w/z)}\psi^-(w)\psi^+(z),\\
 \psi^+(z)x^\pm(w)=g(\gamma^{\mp 1/2}w/z)^{\mp1} x^\pm(w)\psi^+(z),
&\psi^-(z)x^\pm(w)=g(\gamma^{\mp 1/2}z/w)^{\pm1} x^\pm(w)\psi^-(z),\\
\multicolumn{2}{c}{\parbox{15cm}{
$[x^+(z),x^-(w)]=\dfrac{(1-q)(1-1/t)}{1-q/t}
\big( \delta(\gamma^{-1}z/w)\psi^+(\gamma^{1/2}w)-
\delta(\gamma z/w)\psi^-(\gamma^{-1/2}w) \big)$,}}\\
G^{\mp}(z/w)x^\pm(z)x^\pm(w)=G^{\pm}(z/w)x^\pm(w)x^\pm(z).
\end{array}
\end{align*}
\end{dfn}

\begin{fct}[{\cite[Proposition A.2]{FHHSY:2009}}]
The algebra $\calU(q,t)$ has a Hopf algebra structure with
\\
Coproduct $\Delta$:
\begin{align*}
\begin{array}{ll}
 \Delta(\gamma^{\pm 1/2})=\gamma^{\pm 1/2} \otimes \gamma^{\pm 1/2},
&
 \Delta (x^+(z))=
x^+(z)\otimes 1+
\psi^-(\gamma_{(1)}^{1/2}z)\otimes x^+(\gamma_{(1)}z),
\\
\Delta (\psi^\pm(z))=
\psi^\pm (\gamma_{(2)}^{\pm 1/2}z)\otimes \psi^\pm (\gamma_{(1)}^{\mp 1/2}z),
&\Delta (x^-(z))=
x^-(\gamma_{(2)}z)\otimes \psi^+(\gamma_{(2)}^{1/2}z)+1 \otimes x^-(z),
\end{array}
\end{align*}
where $\gamma_{(1)}^{\pm 1/2}=\gamma^{\pm 1/2}\otimes 1$
and $\gamma_{(2)}^{\pm 1/2}=1\otimes \gamma^{\pm 1/2}$.\\
Counit $\ve$:
\begin{align*}
\ve(\gamma^{\pm 1/2})=1,\qquad \ve(\psi^\pm(z))=1,
\qquad \ve(x^\pm(z))=0.
\end{align*}
Antipode $a$:
\begin{align*}
\begin{array}{ll}
 a(\gamma^{\pm 1/2})=\gamma^{\mp 1/2},
&a(x^+(z))=-\psi^-(\gamma^{-1/2}z)^{-1}x^+(\gamma^{-1} z),
\\
 a(\psi^{\pm}(z))=\psi^{\pm}(z)^{-1},
&a(x^-(z))=-x^-(\gamma^{-1} z)\psi^+(\gamma^{-1/2}z)^{-1}.
\end{array}
\end{align*}
\end{fct}

\subsection{Level one representation of $\calU(q,t)$}

We say a representation of $\calU(q,t)$ is of level $k$, 
if the central element $\gamma$ is realized by the constant 
$(t/q)^{k/2}=p^{-k/2}$.

\begin{fct}[{\cite[Proposition A.6]{FHHSY:2009}}]\label{fct:level1rep}
Consider the Heisenberg Lie algebra $\frakh$ over $\bbF$ 
with the generators $a_n$ ($n\in\bbZ$) and the relations
\begin{align}\label{eq:boson_macdonald}
[a_m,a_n] = m\dfrac{1-q^{|m|}}{1-t^{|m|}}\delta_{m+n,0}\, a_0.
\end{align}
Let $\frakh^{\ge0}$ (resp. $\frakh^{<0}$) be the subalgebra 
generated by $a_n$ for $n\ge 0$ (resp. $n< 0$). 
Consider the one dimensional representation $\widetilde{\calF}$ 
of $\frakh^{\ge0}$, where $a_n$ ($n>0$) acts trivially and $a_0$ 
acts by some fixed element of $\widetilde{\calF}$.
Then one has the induced Fock representation 
$\calF\seteq\Ind_{\frakh^{\ge0}}^\frakh\widetilde{\bbF}$ of $\frakh$.
Let us also introduce the following four vertex operators \cite[(1.7),(3.23),(3.27),(3.28)]{FHHSY:2009}.
\begin{align*}
&\eta(z)\seteq 
\exp\Big( \sum_{n>0} \dfrac{1-t^{-n}}{n}a_{-n} z^{n} \Big)
\exp\Big(-\sum_{n>0} \dfrac{1-t^{n} }{n}a_n    z^{-n}\Big),
\\
&\xi(z) \seteq
\exp\Big(-\sum_{n>0} \dfrac{1-t^{-n}}{n}p^{-n/2}a_{-n} z^{n}\Big)
\exp\Big(\sum_{n>0}  \dfrac{1-t^{n}}{n} p^{-n/2} a_n z^{-n}\Big),
\\
&\varphi^+(z) \seteq
\exp\Big(-\sum_{n>0} \dfrac{1-t^{n}}{n} (1-p^{-n})p^{n/4} a_n z^{-n}
    \Big),
\\
&\varphi^-(z) \seteq
\exp\Big(\sum_{n>0} \dfrac{1-t^{-n}}{n}(1-p^{-n})p^{n/4} a_{-n} z^{n}
     \Big).
\end{align*}
Then for a fixed $c\in \widetilde{\bbF}^{\times}$,
we have a level one representation $\rho_c(\cdot)$ of $\calU(q,t)$ 
on $\calF$ by setting
\begin{align*}
\rho_c(\gamma^{\pm 1/2})=p^{\mp 1/4},\quad
\rho_c(\psi^\pm(z))=\varphi^\pm(z),\quad 
\rho_c(x^+(z))=c\, \eta(z),\quad
\rho_c(x^-(z))=c^{-1} \xi(z).
\end{align*}
\end{fct}

\begin{rmk}\label{rmk:level1rep}
We can rephrase this fact as follows.
Let us define $b_n$'s by the expansion of $\psi^{\pm}$:
\begin{align}\label{eq:psi_boson}
\psi^+(z)=\psi^+_0 
 \exp\left(+\sum_{n>0} b_n \gamma^{n/2} z^{-n} \right) ,
\quad
\psi^-(z)=\psi^-_0 
 \exp\left(-\sum_{n>0} b_{-n} \gamma^{n/2}z^{n} \right).
\end{align}
Then we have
\begin{align}\label{eq:boson_b}
[b_m,b_n]
=\dfrac{1}{m}(1-q^{-m})(1-t^m)(1-p^m)(\gamma^m-\gamma^{-m})
\gamma^{-|m|}\delta_{m+n,0},
\end{align}
and the coproduct for $b_n$ reads
\begin{align}\label{eq:Delta_boson}
\Delta(b_n)=b_n\otimes \gamma^{-|n|}+1\otimes b_n.
\end{align}
Then the representation $\rho_c$ is given by 
$\gamma^{\pm 1/2} \mapsto p^{\mp 1/4}$ and
\begin{align*}
&b_n    \mapsto -\dfrac{1-t^n}{|n|}(p^{|n|/2}-p^{-|n|/2})a_n,
\quad
\psi_0^{\pm}\mapsto 1,\quad 
x^+(z) \mapsto c\, \eta(z),\quad
x^-(z)\mapsto c^{-1} \xi(z).
\end{align*}
\end{rmk}

\begin{dfn}
Consider the $m$-fold tensor representation 
$\rho_{y_1}\otimes\cdots \otimes\rho_{y_m}$ on $\calF^{\otimes m}$
for $m\in\bbZ_{\geq 2}$.
Define $\Delta^{(m)}$ inductively by
\begin{align*}
\Delta^{(2)}\seteq \Delta,\quad
\Delta^{(m)}\seteq({\rm id}\otimes \cdots 
\otimes{\rm id}\otimes \Delta)\circ \Delta^{(m-1)}.
\end{align*} 
Since we have 
$\rho_{y_1}\otimes\cdots \otimes\rho_{y_m}\Delta^{(m)}(\gamma)=
\gamma_{(1)}\cdots \gamma_{(m)}=p^{-m/2}$, the level is $m$. 
We also define
\begin{align}\label{eq:rhom}
\rho_y^{(m)}\seteq\rho_{y_1}\otimes\cdots \otimes\rho_{y_m}
\circ\Delta^{(m)} : \calU(q,t)\to\calF^{\otimes m}.
\end{align}
\end{dfn}

\begin{lem}\label{lem:tildeLambda}
We have 
\begin{align*}
\rho^{(m)}_y(x^+(z))=\sum_{i=1}^m y_i      \widetilde{\Lambda}_i(z),
\quad
\rho^{(m)}_y(x^-(z))=\sum_{i=1}^m y_i^{-1} \widetilde{\Lambda}^*_i(z),
\end{align*}
where the $\widetilde{\Lambda}_i(z)$, $\widetilde{\Lambda}_i^*(z)$ 
are defined to be
\begin{align}
\label{eq:tildeL}
\widetilde{\Lambda}_i(z)
&\seteq
 \varphi^-(p^{-1/4}z)\otimes\varphi^-(p^{-3/4}z)\otimes\cdots
 \otimes \varphi^-(p^{-(2i-3)/4}z)
 \otimes \eta(p^{-(i-1)/2}z)\otimes 1\otimes \cdots\otimes 1,
\\
\label{eq:tildeL*}
\widetilde{\Lambda}_i^*(z)
&\seteq
 1\otimes \cdots \otimes 1
  \otimes \xi(p^{-(m-i)/2}z)
  \otimes \varphi^+ (p^{-(2m-2i-1)/4}z)\otimes\cdots
  \otimes\varphi^+(p^{-1/4}z),
\end{align}
where $\eta(p^{-(i-1)/2}z)$ and $\xi(p^{-(m-i)/2}z)$ 
sit in the $i$-th tensor component.
\end{lem}
\begin{proof}
By the definition \eqref{eq:rhom}, Fact \ref{fct:level1rep} 
and Remark \ref{rmk:level1rep}.
\end{proof}

\subsection{New currents $t(z)$ and $t^*(z)$}

\begin{dfn}\label{dfn:tt*}
We define
\begin{align}
\label{eq:tt*}
t(z)   \seteq \alpha(z)             x^{+}(z) \beta(z),\quad
t^*(z) \seteq \alpha(p^{-1} z)^{-1}x^{-}(p^{-1} \gamma^{-1}z) 
              \beta(\gamma^{-2}p^{-1} z)^{-1}.
\end{align}
Here we used auxiliary vertex operators
\begin{align}
\label{eq:alpha_beta}
\alpha(z)   \seteq 
 \exp\Big(-\sum_{n=1}^\infty \dfrac{1}{\gamma^n-\gamma^{-n}}
          b_{-n} z^{n}\Big),
\quad
\beta(z)   \seteq 
 \exp\Big(\sum_{n=1}^\infty \dfrac{1}{\gamma^n-\gamma^{-n}}
          b_{n}z^{-n}\Big).
\end{align}
Here the part $1/(\gamma^n-\gamma^{-n})$ is considered 
to be the formal power sum $\sum_{i=0}^\infty \gamma^{-(2i+1)n}$.
\end{dfn}

\begin{rmk}\label{rmk:t*}
The definition of $t^*(z)$ can be read as
\begin{align*}
t^*(\gamma p z) = \alpha(\gamma z)^{-1}x^{-}(z) \beta(\gamma^{-1}z)^{-1}.
\end{align*}
This form is convenient in the actual calculations.
\end{rmk}

\begin{prp}\label{prp:tt*}
(1) The elements $t(z)$ and $t^*(z)$ commutes with 
$\alpha(w)$, $\beta(w)$ and $\psi^{\pm}(w)$:
\begin{align*}
&[t(z),\alpha(w)]=[t(z),\beta(w)]=[t^*(z),\alpha(w)]=[t^*(z),\beta(w)]=0,\\
&[t(z),\psi^{\pm}(w)]=[t^*(z),\psi^{\pm}(w)]=0.
\end{align*}

(2)
Set
\begin{align}
\label{eq:A}
A(z)\seteq \exp\Big(\sum_{n=1}^\infty\dfrac{1}{n}\dfrac{(1-q^n)(1-t^{-n})(1-p^{-n} \gamma^{-2n})}{1-\gamma^{-2n}}z^n\Big),
\end{align}
where the part $1/(1-\gamma^{-2n})$ is considered to be the formal power sum 
$\sum_{i=0}^\infty \gamma^{-2i n}$.
Then we have
\begin{align}
\label{eq:Att}
&A(\tfrac{w}{z})t(z)t(w)-A(\tfrac{z}{w})t(w)t(z)
=\dfrac{(1-q)(1-t^{-1})}{1-p}
 [\delta(p^{-1}\tfrac{w}{z})t^{(2)}(z)-\delta(p \tfrac{w}{z})t^{(2)}(w)],
\\
\label{eq:Att*}
&A(\tfrac{w}{z})t^*(z)t^*(w)-A(\tfrac{z}{w})t^*(w)t^*(z)
=\dfrac{(1-q^{-1})(1-t)}{1-p^{-1}}
 [\delta(p \tfrac{w}{z})t^{*(2)}(z)-\delta(p^{-1}\tfrac{w}{z})t^{*(2)}(w)],
\end{align}
where $\delta(z)\seteq \sum_{n\in\bbN}z^n + z^{-1}\sum_{n\in\bbN}z^{-n}$ 
is the formal delta function, and 
\begin{align*}
&t^{(2)}(z)\seteq 
\alpha(p z)\alpha(z)x^+(p z)x^+(z)\beta(p z)\beta(z),
\\
&t^{*(2)}(z)\seteq
\alpha(\gamma p z)^{-1}\alpha(\gamma z)^{-1}x^-(p z)x^-(z)
\beta(\gamma^{-1}p z)^{-1} \beta(\gamma^{-1} z)^{-1}.
\end{align*}

(3)
As in (2), set
\begin{align}
\label{eq:B}
B(z)\seteq 
\exp\Big(\sum_{n=1}^\infty\dfrac{1}{n}
         \dfrac{(1-q^n)(1-t^{-n})(p^{-2n}  \gamma^{-2n} -p^{-n} \gamma^{-2n} )}
               {1-\gamma^{-2n}} z^n
     \Big).
\end{align}
Then 
\begin{align}
\label{eq:Btt*}
B(\tfrac{w}{z})t(z)t^*(w)-B(\gamma^2 p^2\tfrac{z}{w})t^*(w)t(z)
=\dfrac{(1-q)(1-t^{-1})}{1-p}
 \Big(\delta(p^{-1}\tfrac{w}{z})\psi_0^+
     -\delta(\gamma^{-2}p^{-1}\tfrac{w}{z})\psi_0^-)\Big).
\end{align}
\end{prp}

\begin{proof}
See \S \ref{subsec:prf:tt*}.
\end{proof}

In the next subsection we show that the currents $t(z)$, $t^*(z)$ are 
connected to the realization of deformed $\calW$ algebra 
in the Fock representation of $\calU(q,t)$.

\subsection{Deformed algebra $\calW_{q,p}(\Liesl_m)$}
We basically follow the description of $\calW_{q,p}(\Liesl_m)$ 
in \cite[\S 4]{FF:1996}. As for the connection between the singular vectors of the
$\calW_{q,p}(\Liesl_m)$ and the Macdonald polynomials, 
see \cite{SKAO:1996,AKOS:1996}.

\begin{dfn}Set
\begin{align*}
f_{k,l}(z)\seteq
\exp\Big(\sum_{n=1}^\infty 
         \dfrac{(1-q^n)(1-t^{-n})(p^{(k-1)n}-p^{(l-1)n})}{1-p^{l n}}z^n \Big).
\end{align*} 
\end{dfn}

\begin{rmk}\label{rmk:ABf}
Our functions $A(z)$ and $B(z)$ give special cases of this function 
under $\rho_y^{(m)}$, that is,
\begin{align*}
\rho_y^{(m)}(A(z))=f_{1,m}(z),\quad 
\rho_y^{(m)}(B(z))=f_{m-1,m}(z).
\end{align*}
\end{rmk}

\begin{dfn}
Set
\begin{align}
\label{eq:TT*}
&T(z)=T_1(z)\seteq \rho^{(m)}_y(t(z)),\quad 
 T^*(z)=T_1^*(z)\seteq \rho^{(m)}_y(t^*(z)).
\end{align}
Let us also define
\begin{align}
\label{eq:LL*}
\Lambda_i(z) \seteq
 \rho^{(m)}_y(\alpha(z)) \widetilde{\Lambda}_i(z) 
 \rho^{(m)}_y(\beta(z)),
\quad
\Lambda_i^*(z) \seteq
 \rho^{(m)}_y(\alpha(p^{-1} z)^{-1}) \widetilde{\Lambda}^*_i(p^{(m-2)/2}z) 
 \rho^{(m)}_y(\beta(\gamma^{-2}p^{-1} z)^{-1}).
\end{align}
Then by Definition \ref{dfn:tt*} and Lemma \ref{lem:tildeLambda} we have
\begin{align}\label{eq:TL}
T_1(z)=\sum_{i=1}^m y_i \Lambda_i(z),
\quad
T^*_1(z)=\sum_{i=1}^m y_i^{-1} \Lambda_i^*(z).
\end{align}
For $i=2,\ldots,m$, we further define
\begin{align}
\label{eq:TiL}
&T_i(z) \seteq \sum_{1\le j_1<\cdots<j_i\le m}
 y_{j_1}y_{j_2}\cdots y_{j_i}
 :\Lambda_{j_1}(z)\Lambda_{j_2}(z p)\cdots\Lambda_{j_i}(z p^{i-1}):,
\\
\label{eq:TiL*}
&T_i^*(z) \seteq \sum_{1\le j_1<\cdots<j_i\le m}
 y_{j_1}^{-1}y_{j_2}^{-1}\cdots y_{j_i}^{-1}
:\Lambda_{j_1}^*(z)\Lambda_{j_2}^*(z p^{-1})\cdots\Lambda_{j_i}^*(z p^{-i+1}):.
\end{align}
\end{dfn}

\begin{prp}\label{prp:Lambda}
(1)
The operator product of $\Lambda_i(z)$ and $\Lambda_j(w)$ is given by
\begin{align}
\label{eq:prp:Lambda:1}
f_{1,m}(\tfrac{w}{z})\Lambda_i(z)\Lambda_j(w)
=:\Lambda_i(z)\Lambda_j(w):\times
\begin{cases}
1&i=j,\\
\gamma_+(z,w;q,p)&i<j,\\
\gamma_-(z,w;q,p)&i>j.
\end{cases}
\end{align}
Here we used the symbol
\begin{align}\label{eq:gamma+-}
\gamma_+(z,w;q,t) \seteq
\dfrac{(z-q^{-1} w)(z-q t^{-1}w)}{(z-w)(z-t^{-1}w)},\quad
\gamma_-(z,w;q,t) \seteq
\dfrac{(z-q w)(z-q^{-1} t w)}{(z-w)(z-t w)}.
\end{align}

(2)
We have
\begin{align*}
:\Lambda_1(z)\Lambda_2(p z)\cdots\Lambda_m(p^{m-1}z): =1.
\end{align*}
Therefore $T_m(z)=y_1y_2\cdots y_m$.

(3)
The $\Lambda_i(z)$ and $\Lambda_j^*(z)$ are connected by the following 
equation.
\begin{align}
\label{eq:prp:Lambda:3}
\Lambda_k^*(z)=
:\prod_{i=1}^{k-1}\Lambda_i(p^{k-1}z)\prod_{l=k+1}^m \Lambda_l(p^{l-2}z):.
\end{align}
Thus we also have
\begin{align}
\label{eq:prp:Lambda:3:2}
T^*_1(z)=y_1^{-1}y_2^{-1}\cdots y_m^{-1}T_{m-1}(z).
\end{align}

(4)
The operator product of $\Lambda_i^*(z)$ and $\Lambda_j^*(w)$ is given by
\begin{align}
\label{eq:prp:Lambda:4}
f_{1,m}(\tfrac{w}{z})\Lambda_i^*(z)\Lambda_j^*(w)
=:\Lambda_i^*(z)\Lambda_j^*(w):\times
\begin{cases}
1&i=j,\\
\gamma_-(z,w;q,p)&i<j,\\
\gamma_+(z,w;q,p)&i>j.
\end{cases}
\end{align}

(5)
We have
\begin{align*}
:\Lambda_1^*(z)\Lambda_2^*(p^{-1} z)\cdots\Lambda_m^*(p^{-m+1}z): =1.
\end{align*}
Therefore $T^*_m(z)=y_1^{-1}y_2^{-1}\cdots y_m^{-1}$.
\end{prp}

\begin{proof}
See \S \ref{subsec:prf:Lambda}.
\end{proof}

\begin{prp}\label{prp:TT}
We have
\begin{align}
\label{eq:TT:1i}
&f_{1,m}(\tfrac{w}{z})T_1(z)T_i(w)-f_{1,m}(p^{1-i}\tfrac{z}{w})T_i(w)T_1(z)
=\dfrac{(1-q)(1-t^{-1})}{1-p}
 [\delta(p^{-1}\tfrac{w}{z}) T_{i+1}(z)-\delta(p^i\tfrac{w}{z}) T_{i+1}(w)],
\\
\label{eq:TT:mm}
\begin{split}
&f_{1,m}(\tfrac{w}{z})T_{m-1}(z)T_{m-1}(w)
-f_{1,m}(\tfrac{z}{w})T_{m-1}(w)T_{m-1}(z)
\\
&\hskip10em
 =\dfrac{(1-q^{-1})(1-t)}{1-p^{-1}}
 [\delta(p \tfrac{w}{z})T_2^*(z)-\delta(p^{-1}\tfrac{w}{z})T_2^*(w)].
\end{split}
\end{align}
\end{prp}

\begin{proof}
\eqref{eq:TT:1i} follows from \eqref{eq:TL}, \eqref{eq:TiL} 
and \eqref{eq:prp:Lambda:1}.
See \cite[Theorem 2]{FF:1996}  for detail\footnote{
It seems that \cite{FF:1996} contains some typo.
In (6.2) of that paper, the term $f_{m,N}(\frac{z}{w})$ should be 
$f_{m,N}(p^{1-m}\frac{z}{w})$.}.

\eqref{eq:TT:mm} is also shown by the same method 
using \eqref{eq:prp:Lambda:3:2}, \eqref{eq:TiL*} and 
and \eqref{eq:prp:Lambda:4}.
\end{proof}

\subsection{Deformed $\calW$ algebra and kernel function}
Our final consequence of this paper relates the vacuum expectation values 
of the deformed algebra $\calW_{q,p}$  with the finite kernel function.
\begin{thm}
Let $\left|0\right>$ be the vacuum of $\calF$, that is, 
$a_0\left|0\right>=\left|0\right>$ and $a_n\left|0\right>=0$ for $n>0$.
Let $\left<0\right|$ to be the dual vacuum.
We denote the tensor $\left|0\right>^{\otimes m}\in\calF^{\otimes m}$ 
by the same symbol $\left|0\right>$.
We use the similar abbreviation for the tensored dual vacuum.
Then, denoting $y=(y_1,\ldots,y_m)$, we have
\begin{align*}
\dfrac{(-1)^n}{(1-q)^n n!}
\prod_{i<j}f_{1,m}(z_i/z_j)
\left<0\right| T_1(z_1)T_1(z_2)\cdots T_1(z_n)\left|0\right>
=K_n(y,z;q,p).
\end{align*}
\end{thm}

\begin{proof}
This follows from \eqref{eq:TL}, 
the operator product \eqref{eq:prp:Lambda:1} 
and the definition \eqref{eq:K:dfn}.
\end{proof}

\section{Proofs of the propositions}

\subsection{Proof of Proposition \ref{prp:b-psi}}\label{subsec:prf:b-psi}

Using the $\gamma_{\pm}$ defined in \eqref{eq:gamma+-}, we have 
\begin{align}\label{eq:oeg}
&\dfrac{\omega(z,w)}{\ep_2^{(t)}(z,w)}=\gamma_+(z,w;q,t),\quad
 \dfrac{\omega(w,z)}{\ep_2^{(t)}(w,z)}=\gamma_-(z,w;q,t),
\quad
 \dfrac{\omega(w,z)}{\omega(z,w)}
=\dfrac{\gamma_-(z,w;q,t)}{\gamma_+(z,w;q,t)}.
\end{align}

For later purpose, we prepare the following formulae.
Let $\theta$ and $\rho$ be natural numbers. 
Then
\begin{align}
\prod_{1\le i<j\le \theta}
\gamma_+(q^{-i}z,q^{-j}w;q,t)
&=\left(\dfrac{1-tz/w}{1-q z/w}\right)^{\theta}
  \dfrac{(q z/w)_\theta}{(t z/w)_\theta}
\label{eq:+:tri}
\\
\prod_{1\le i<j\le \theta}
\gamma_-(q^{-i}z,q^{-j}w;q,t)
&=\left(\dfrac{1-z/w}{1-q t^{-1}z/w}\right)^{\theta}
  \dfrac{(q t^{-1}z/w)_\theta}{(z/w)_\theta}
\label{eq:-:tri}
\\
\prod_{l=1}^\theta \prod_{k=1}^\rho
\gamma_+(q^{-l}z,q^{-k}w;q,t)
&=\dfrac{(q^{-\rho}w/z)_\theta}{(w/z)_\theta}
 \dfrac{(q t^{-1}w/z)_\theta}{(q^{-\rho+1}t^{-1}w/z)_\theta}
\label{eq:+:1}
\\
&=\dfrac{(q^{\rho-\theta+1} z/w)_\theta}{(q^{-\theta+1} z/w)_\theta}
 \dfrac{(q^{-\theta} t z/w)_\theta}{(q^{\rho-\theta}t z/w)_\theta},
\label{eq:+:2}
\\
\prod_{l=1}^\theta \prod_{k=1}^\rho
\gamma_-(q^{-l}z,q^{-k}w;q,t)
&=\dfrac{(q w/z)_\theta}{(q^{-\rho+1}w/z)_\theta}
 \dfrac{(q^{-\rho} t w/z)_\theta}{(t w/z)_\theta}
\label{eq:-:1}
\\
&=\dfrac{(q^{-\theta} z/w)_\theta}{(q^{\rho-\theta}z/w)_\theta}
  \dfrac{(q^{\rho-\theta+1} t^{-1} z/w)_\theta}
        {(q^{-\theta+1}t^{-1} z/w)_\theta}.
\label{eq:-:2}
\end{align}
Here we used $(u)_n\seteq(u;q)_n=\prod_{i=1}^n(1-u q^{i-1})$.
These equations are checked by simple calculations.


\subsubsection{Proof of \eqref{eq:b}} 

By \eqref{eq:fq} we have
\begin{align*}
\dfrac{(1-q)^n n!}{(-1)^n}F_\lambda(z;q,t)=
\left(\dfrac{1-q}{1-t}\right)^{|\lambda|}
\sum_{\mu\ge\lambda'}c_{\lambda\mu}^{e\to P}(q,t)\ep_{\mu}(z;q)
\dfrac{|\mu|!}{\prod_{i=1}^{\ell(\mu)}\mu_i!}.
\end{align*}
Recalling the argument of \cite[Proposition 2.19]{FHHSY:2009},
we find that under the specialization $\varphi^{(q^{-1})}_\lambda$
only the term $\ep_{\lambda'}$ in $F_\lambda(z;q,t)$  survives 
and the other terms $\ep_{\mu}$ vanish.
The specialization result is
\footnote{This expression is given at the last equation in the proof of \cite[Proposition 2.19]{FHHSY:2009}, although it contains a typo. The range  ``$1\le j< k\le l$" of the third product should be ``$1\le j<k\le \ell(\lambda')$"}
\begin{align*}
\varphi^{(q^{-1})}_\lambda\ep_{\lambda'}(y)
&=
\dfrac{\prod_{h=1}^{\ell(\lambda)}\lambda'_h!}{n!}
 \prod_{i=1}^{\ell(\lambda)} \ep_{\lambda'_i}(y_1,\ldots,y_{\lambda'_i};q)
 \prod_{1\le j<k\le \ell(\lambda')}
 \prod_{\alpha=1}^{\lambda_j'}
 \prod_{\beta=1}^{\lambda_k'}\omega(q^{-j+1} y_\alpha,q^{-k+1} y_\beta)\\
&=
\dfrac{\prod_{h=1}^{\ell(\lambda)}\lambda'_h!}{n!}
 \prod_{i=1}^{\ell(\lambda)} \ep_{\lambda'_i}(y_1,\ldots,y_{\lambda'_i};q)
\times
 \prod_{\alpha=1}^{\ell(\lambda)}
 \prod_{1\le i<j\le \lambda_\alpha}
 \omega(q^{-i+1} y_\alpha,q^{-j+1} y_\alpha)
\\
&\phantom{=}\times
 \prod_{1\le \alpha<\beta\le \ell(\lambda)}
 \Big[\prod_{1\le i<j\le \lambda_\beta}
 \omega(q^{-i+1} y_\alpha,q^{-j+1} y_\beta)
 \omega(q^{-i+1} y_\beta,q^{-j+1} y_\alpha)
\\
&\phantom{=\times \prod_{1\le \alpha<\beta\le \ell(\lambda)}\Big[}
 \times
 \prod_{i=1}^{\lambda_\beta}\prod_{j=\lambda_\beta+1}^{\lambda_\alpha}
 \omega(q^{-i+1} y_\beta,q^{-j+1} y_\alpha)
 \Big].
\end{align*}
We also note that $c_{\lambda'\lambda}^{e\to P}(q,t)=1$.

Recalling \eqref{eq:gamma-omega}, 
we can also calculate the first specialization $\varphi^{(q^{-1})}_\lambda$
of the numerator in \eqref{eq:b} as 
\begin{align*}
&\varphi^{(q^{-1})}_\lambda \gamma_{D}(z)
=\Big[\prod_{k=1}^{\ell(\lambda)}
  \prod_{1\le i<j\le \lambda_k}\ep_2(q^{-i},q^{-j};t)\Big]
  \Big[\prod_{\alpha=1}^{\ell(\lambda)}\prod_{\beta=\alpha}^{\ell(\lambda)}
  \prod_{i=1}^{\lambda_\alpha}\prod_{j=1}^{\lambda_\beta}
  \omega(q^{-i}y_\alpha,q^{-j}y_\beta)\Big]
\\
&=\Big[\prod_{k=1}^{\ell(\lambda)}
  \prod_{1\le i<j\le \lambda_k}\ep_2(q^{-i},q^{-j};t)\Big]
 \prod_{1\le \alpha<\beta\le \ell(\lambda)}\Big[
 \big(
  \prod_{1\le i<j\le \lambda_\beta} \omega(q^{-i}y_\alpha,q^{-j}y_\beta)
 \big)
 \big(
  \prod_{1\le j<i\le \lambda_\beta} \omega(q^{-i}y_\alpha,q^{-j}y_\beta)
 \big)
\\
&\phantom{=\times\prod_{1\le \alpha<\beta\le \ell(\lambda)}\Big[}
 \big(
  \prod_{j=1}^{\lambda_\beta}\prod_{i=\lambda_\beta+1}^{\lambda_\alpha} 
  \omega(q^{-i}y_\alpha,q^{-j}y_\beta)
 \big)
 \big(
  \prod_{i=1}^{\lambda_\alpha} \omega(q^{-i}y_\alpha,q^{-i}y_\beta)
 \big)
 \Big].
\end{align*}

Thus we have
\begin{align*}
&
\dfrac{(-1)^n}{(1-q)^n n!}
\widetilde{\varphi}^{(\zeta)}_\lambda 
\dfrac{\gamma_{D}(z)}{F_\lambda(z;q,t)}
=
\left(\dfrac{1-t}{1-q}\right)^{|\lambda|}
\prod_{\alpha=1}^{\ell(\lambda)} 
\prod_{1\le i<j\le\lambda_\alpha}
\dfrac{\ep_2^{(t)}(q^{-i+1}y_\alpha,q^{-j+1}y_\beta)}
      {\omega(q^{-i+1}y_\alpha,q^{-j+1}y_\beta)}\times
\\
&
\prod_{1\le\alpha<\beta\le\ell(\lambda)}\Big[
\big(\prod_{1\le i<j\le\lambda_\beta}
\dfrac{\omega(q^{-i+1}y_\alpha,q^{-j+1}y_\beta)}
      {\omega(q^{-i+1}y_\beta,q^{-j+1}y_\alpha)}\big)
\big(\prod_{i=1}^{\lambda_\beta}\prod_{j=\lambda_\beta+1}^{\lambda_\alpha}
\dfrac{\omega(q^{-j+1}y_\alpha,q^{-i+1}y_\beta)}
      {\omega(q^{-i+1}y_\beta,q^{-j+1}y_\alpha)}\big)
\big(
\dfrac{\omega(y_\alpha,y_\beta)}
      {\ep_2^{(q)}(y_\beta, y_\alpha)}\big)^{\lambda_\beta}
\Big].
\end{align*}

Then recalling \eqref{eq:oeg} and using \eqref{eq:+:tri} and \eqref{eq:-:tri},
one has
\begin{align*}
&
\prod_{1\le i<j\le\lambda_\alpha}
\dfrac{\ep_2^{(t)}(q^{-i+1}y_\alpha,q^{-j+1}y_\beta)}
      {\omega(q^{-i+1}y_\alpha,q^{-j+1}y_\beta)}
=
\left(\dfrac{1-q}{1-t}\right)^{\lambda_\alpha}
\dfrac{(t)_{\lambda_\alpha}}{(q)_{\lambda_\alpha}},
\\
&
\prod_{1\le i<j\le\lambda_\beta}
\dfrac{\omega(q^{-i+1}y_\alpha,q^{-j+1}y_\beta)}
      {\omega(q^{-i+1}y_\beta,q^{-j+1}y_\alpha)}
\Big[\dfrac{\omega(y_\alpha,y_\beta)}
      {\ep_2^{(q)}(y_\beta, y_\alpha)}\Big]^{\lambda_\beta}
=
\dfrac{(t y_\beta/y_\alpha)_{\lambda_\alpha}}
      {(q y_\beta/y_\alpha)_{\lambda_\alpha}}
\dfrac{(q t^{-1} y_\beta/y_\alpha)_{\lambda_\alpha}}
      {(y_\beta/y_\alpha)_{\lambda_\alpha}},
\\
&
\prod_{i=1}^{\lambda_\beta}\prod_{j=\lambda_\beta+1}^{\lambda_\alpha}
\dfrac{\omega(q^{-j+1}y_\alpha,q^{-i+1}y_\beta)}
      {\omega(q^{-i+1}y_\beta,q^{-j+1}y_\alpha)}
\\
&=
\dfrac{(q y_\beta/y_\alpha)_{\lambda_\beta}}
      {(t y_\beta/y_\alpha)_{\lambda_\beta}}
\dfrac{(  y_\beta/y_\alpha)_{\lambda_\beta}}
      {(q t^{-1} y_\beta/y_\alpha)_{\lambda_\beta}}
\dfrac{(q^{\lambda_\alpha-\lambda_\beta}t y_\beta/y_\alpha)_{\lambda_\beta}}
      {(q^{\lambda_\alpha-\lambda_\beta}  y_\beta/y_\alpha)_{\lambda_\beta}}
\dfrac{(q^{\lambda_\alpha-\lambda_\beta+1}t^{-1} 
        y_\beta/y_\alpha)_{\lambda_\beta}}
      {(q^{\lambda_\alpha-\lambda_\beta+1}y_\beta/y_\alpha)_{\lambda_\beta}}.
\end{align*}

Combining these factors, we obtain
\begin{align*}
\dfrac{(-1)^n}{(1-q)^n n!}
\lim_{\zeta\to t}\widetilde{\varphi}^{(\zeta)}_\lambda
\dfrac{\gamma_{D}(z;t)}{F_\lambda(z;q,t)}
&=
\prod_{\alpha=1}^{\ell(\lambda)}
\dfrac{(t)_{\lambda_\alpha}}{(q)_{\lambda_\alpha}}
\prod_{1\le\alpha<\beta\le\ell(\lambda)}
\dfrac{(q^{\lambda_\alpha-\lambda_\beta}t y_\beta/y_\alpha)_{\lambda_\beta}}
      {(q^{\lambda_\alpha-\lambda_\beta}  y_\beta/y_\alpha)_{\lambda_\beta}}
\dfrac{(q^{\lambda_\alpha-\lambda_\beta+1}t^{-1} 
        y_\beta/y_\alpha)_{\lambda_\beta}}
      {(q^{\lambda_\alpha-\lambda_\beta+1}y_\beta/y_\alpha)_{\lambda_\beta}}
\\
&=
\prod_{1\le\alpha<\beta\le\ell(\lambda)}
\dfrac{(q^{\lambda_\alpha-\lambda_\beta}t^{\beta-\alpha+1})_{\lambda_\beta-\lambda_{\beta+1}}}
      {(q^{\lambda_\alpha-\lambda_\beta+1}t^{\beta-\alpha})_{\lambda_\beta-\lambda_{\beta+1}}}.
\end{align*}
But one can easily find that the last expression equals to $b_\lambda(q,t)$ 
using the form \eqref{eq:b_fact}.

\subsubsection{Proof of \eqref{eq:psi}} 

For a tableau $T\in\Tbr{\lambda;m}$,
define $\theta_{\alpha,k}$ and $\lambda_\alpha^{(k)}$ 
as explained in \S \ref{subsec:tableau}.
Then by the direct calculation we have
\begin{align}
\widetilde{\varphi}^{(\zeta)}_\lambda \dfrac{\gamma_{T}(z)}{\gamma_{D}(z)}
=
&
\prod_{1\le\alpha<\beta\le\ell(\lambda)}
\prod_{i=1}^{\lambda_\alpha}
\prod_{j=1}^{\lambda_\beta}
\gamma_+(q^{-i}\zeta^\alpha,q^{-j}\zeta^\beta)^{-1}
\label{eq:psi:1}
\\
&\times
\prod_{k=1}^m
\prod_{\alpha=1}^{\ell(\lambda)}
\prod_{\beta=\alpha+1}^{\ell(\lambda)}
\prod_{i=1}^{\theta_{\alpha,k}}
\prod_{j=1}^{\lambda_\beta^{(k-1)}}
\gamma_-(q^{-i-\lambda_\alpha^{(k-1)}}\zeta^\alpha,q^{-j}\zeta^\beta)
\label{eq:psi:2}
\\
&\times
\prod_{k=1}^m
\prod_{\alpha=1}^{\ell(\lambda)}
\prod_{\beta=\alpha}^{\ell(\lambda)}
\prod_{i=1}^{\theta_{\alpha,k}}
\prod_{j=1+\lambda_\beta^{(k)}}^{\lambda_\beta}
\gamma_+(q^{-i-\lambda_\alpha^{(k-1)}}\zeta^\alpha,q^{-j}\zeta^\beta)
\label{eq:psi:3}
\end{align}

By the formula \eqref{eq:+:1} we find that
\begin{align}\label{eq:psi:1:2}
\lim_{\zeta\to t}\eqref{eq:psi:1}=
\prod_{1\le\alpha<\beta\le\ell(\lambda)}
\dfrac{(t^{\beta-\alpha})_{\lambda_\alpha}}
      {(q^{-\lambda_\beta}t^{\beta-\alpha})_{\lambda_\alpha}}
\dfrac{(q^{-\lambda_\beta+1}t^{\beta-\alpha-1})_{\lambda_\alpha}}
      {(q t^{\beta-\alpha-1})_{\lambda_\alpha}}.
\end{align}
Note that the regularity of \eqref{eq:psi:1} at $\zeta=t$ 
is included in this equation.
Similarly by the formula \eqref{eq:-:1}, 
\eqref{eq:psi:2} is regular at $\zeta=t$ and its value is
\begin{align}\label{eq:psi:2:2}
\lim_{\zeta\to t}\eqref{eq:psi:2}=
\prod_{k=1}^m
\prod_{\alpha=1}^{\ell(\lambda)}
\prod_{\beta=\alpha+1}^{\ell(\lambda)}
\dfrac{(q^{\lambda_\alpha^{(k-1)}+1}t^{\beta-\alpha})_{\theta_{\alpha,k}}}
      {(q^{\lambda_\alpha^{(k-1)}-\lambda_\beta^{(k-1)}}
       t^{\beta-\alpha+1})_{\theta_{\alpha,k}}}
\dfrac{(q^{\lambda_\alpha^{(k-1)}-\lambda_\beta^{(k-1)}}
       t^{\beta-\alpha+1})_{\theta_{\alpha,k}}}
      {(q^{\lambda_\alpha^{(k-1)}-\lambda_\beta^{(k-1)}+1}
       t^{\beta-\alpha})_{\theta_{\alpha,k}}}.
\end{align}
The rest term \eqref{eq:psi:3} is calculated by the formula \eqref{eq:+:1} and \eqref{eq:+:2}:
\begin{align}\label{eq:psi:3:2}
\begin{split}
\lim_{\zeta\to t}\eqref{eq:psi:3}=
\prod_{k=1}^m
\prod_{\alpha=1}^{\ell(\lambda)}\Big[
&\dfrac{(t)_{\theta_{\alpha,k}}}{(q)_{\theta_{\alpha,k}}}
\dfrac{(q^{\lambda_\alpha-\lambda_\alpha^{(k)}+1})_{\theta_{\alpha,k}}}
      {(q^{\lambda_\alpha-\lambda_\alpha^{(k)}})_{\theta_{\alpha,k}}}
\\
&\prod_{\beta=\alpha+1}^{\ell(\lambda)}
\dfrac{(q^{\lambda_\alpha^{(k-1)}-\lambda_\beta^{(k-1)}+1}
        t^{\beta-\alpha-1})_{\theta_{\alpha,k}}}
      {(q^{\lambda_\alpha^{(k-1)}-\lambda_\beta^{(k)}}
       t^{\beta-\alpha})_{\theta_{\alpha,k}}}
\dfrac{(q^{\lambda_\alpha^{(k-1)}-\lambda_\beta}
       t^{\beta-\alpha})_{\theta_{\alpha,k}}}
      {(q^{\lambda_\alpha^{(k-1)}-\lambda_\beta+1}
       t^{\beta-\alpha-1})_{\theta_{\alpha,k}}}
\Big].
\end{split}
\end{align}
Note that some parts of \eqref{eq:psi:2:2} and \eqref{eq:psi:3:2} are 
combined into the next form.
\begin{align*}
&
\Big[\prod_{k=1}^m
\prod_{\alpha=1}^{\ell(\lambda)}
\prod_{\beta=\alpha+1}^{\ell(\lambda)}
\dfrac{(q^{\lambda_\alpha^{(k-1)}+1}t^{\beta-\alpha})_{\theta_{\alpha,k}}}
      {(q^{\lambda_\alpha^{(k-1)}-\lambda_\beta^{(k-1)}}
       t^{\beta-\alpha+1})_{\theta_{\alpha,k}}}\Big]
\\
&
\times
\Big[
\prod_{k=1}^m
\prod_{\alpha=1}^{\ell(\lambda)}
\dfrac{(q^{\lambda_\alpha-\lambda_\alpha^{(k)}+1})_{\theta_{\alpha,k}}}
      {(q^{\lambda_\alpha-\lambda_\alpha^{(k)}})_{\theta_{\alpha,k}}}
\prod_{\beta=\alpha+1}^{\ell(\lambda)}
\dfrac{(q^{\lambda_\alpha^{(k-1)}-\lambda_\beta}
       t^{\beta-\alpha})_{\theta_{\alpha,k}}}
      {(q^{\lambda_\alpha^{(k-1)}-\lambda_\beta+1}
       t^{\beta-\alpha-1})_{\theta_{\alpha,k}}}\Big]
\\
&=
\Big[\prod_{\alpha=1}^{\ell(\lambda)}
 \dfrac{(q)_{\lambda_\alpha}}{(t)_{\lambda_\alpha}}\Big]
\Big[\prod_{1\le\alpha<\beta\le\ell(\lambda)}
 \dfrac{(q^{-\lambda_\beta}t^{\beta-\alpha})_{\lambda_\alpha}}
       {(q^{-\lambda_\beta+1}t^{\beta-\alpha-1})_{\lambda_\alpha}}
 \dfrac{(q t^{\beta-\alpha})_{\lambda_\alpha}}
       {(  t^{\beta-\alpha+1})_{\lambda_\alpha}}\Big]
\\
&=
\Big[\prod_{1\le\alpha<\beta\le\ell(\lambda)}
 \dfrac{(q^{-\lambda_\beta}t^{\beta-\alpha})_{\lambda_\alpha}}
       {(q^{-\lambda_\beta+1}t^{\beta-\alpha-1})_{\lambda_\alpha}}\Big]
\Big[\prod_{1\le\alpha<\beta\le\ell(\lambda)+1}
 \dfrac{(q t^{\beta-\alpha+1})_{\lambda_\alpha}}
       {(  t^{\beta-\alpha})_{\lambda_\alpha}}\Big]
=\eqref{eq:psi:3:2}^{-1}\times 
 \prod_{\alpha=1}^{\ell(\lambda)}
 \dfrac{(q t^{\ell(\lambda)-\alpha})_{\lambda_\alpha}}
       {(t^{\ell(\lambda)-\alpha+1})_{\lambda_\alpha}}.
\end{align*}

Therefore we have
\begin{align}
\lim_{\zeta\to t}
&\widetilde{\varphi}^{(\zeta)}_\lambda \dfrac{\gamma_{T}(z)}{\gamma_D(z)}
=\Big[
 \prod_{\alpha=1}^{\ell(\lambda)}
 \dfrac{(q t^{\ell(\lambda)-\alpha})_{\lambda_\alpha}}
            {(  t^{\ell(\lambda)-\alpha+1})_{\lambda_\alpha}}
      \prod_{k=1}^m\dfrac{(t)_{\theta_{\alpha,k}}}{(q)_{\theta_{\alpha,k}}}
 \Big]\times
\nonumber
\\
&\phantom{==}\prod_{k=1}^m
 \prod_{1\le\alpha<\beta\le\ell(\lambda)}
 \dfrac{(q^{\lambda_\alpha^{(k-1)}-\lambda_\beta^{(k-1)}}
         t^{\beta-\alpha+1})_{\theta_{\alpha,k}}}
       {(q^{\lambda_\alpha^{(k-1)}-\lambda_\beta^{(k-1)}+1}
         t^{\beta-\alpha})_{\theta_{\alpha,k}}}
 \dfrac{(q^{\lambda_\alpha^{(k-1)}-\lambda_\beta^{(k-1)}+1}
         t^{\beta-\alpha+1})_{\theta_{\alpha,k}}}
       {(q^{\lambda_\alpha^{(k-1)}-\lambda_\beta^{(k)}}
         t^{\beta-\alpha})_{\theta_{\alpha,k}}}
\nonumber
\\
&=\prod_{k=1}^m
 \prod_{1\le\alpha\le\beta\le\ell(\lambda)}
 \dfrac{(q^{\lambda_\alpha^{(k-1)}-\lambda_\beta^{(k-1)}}
         t^{\beta-\alpha+1})_{\theta_{\alpha,k}}}
       {(q^{\lambda_\alpha^{(k-1)}-\lambda_\beta^{(k-1)}+1}
         t^{\beta-\alpha})_{\theta_{\alpha,k}}}
 \times
 \prod_{k=1}^m
 \prod_{1\le\alpha\le\beta\le\ell(\lambda)}
 \dfrac{(q^{\lambda_\alpha^{(k-1)}-\lambda_{\beta+1}^{(k)}+1}
         t^{\beta-\alpha})_{\theta_{\alpha,k}}}
       {(q^{\lambda_\alpha^{(k-1)}-\lambda_{\beta+1}^{(k)}+1}
         t^{\beta-\alpha+1})_{\theta_{\alpha,k}}}.
\label{eq:psi:4}
\end{align}
Note that the function $f(u)\seteq (t u)_\infty/(q u)_\infty$ satisfies
$f(u)/f(q^{-\theta}u)=(q^{-\theta+1} u)_\infty/(q^{-\theta} t u)_\infty$.
Then \eqref{eq:psi:4} can be rewritten into 
\begin{align}\label{eq:psi:5}
\eqref{eq:psi:4}
=\prod_{k=1}^m
 \prod_{1\le\alpha\le\beta\le\ell(\lambda)}
 \dfrac{f(q^{\lambda_\alpha^{(k-1)}-\lambda_{\beta}^{(k-1)}}t^{\beta-\alpha})}
       {f(q^{\lambda_\alpha^{(k)}  -\lambda_{\beta}^{(k-1)}}t^{\beta-\alpha})}
 \dfrac{f(q^{\lambda_\alpha^{(k)}-\lambda_{\beta+1}^{(k)}  }t^{\beta-\alpha})}
       {f(q^{\lambda_\alpha^{(k-1)}-\lambda_{\beta+1}^{(k)}}t^{\beta-\alpha})}.
\end{align}

Finally, if $T\in\SSTb{\lambda;m}$, we have $k\ge\ell(\lambda^{(k)})$,
Therefore if $\beta\ge k$ 
then $\lambda_{\beta+1}^{(k)}=\lambda_{\beta}^{(k-1)}=0$.
Thus one can see that 
\begin{align*}
\eqref{eq:psi:5}
&=\prod_{k=1}^m
 \prod_{1\le\alpha\le\beta\le\ell(\lambda^{(k-1)})}
 \dfrac{f(q^{\lambda_\alpha^{(k-1)}-\lambda_{\beta}^{(k-1)}}t^{\beta-\alpha})}
       {f(q^{\lambda_\alpha^{(k)}  -\lambda_{\beta}^{(k-1)}}t^{\beta-\alpha})}
 \dfrac{f(q^{\lambda_\alpha^{(k)}-\lambda_{\beta+1}^{(k)}  }t^{\beta-\alpha})}
       {f(q^{\lambda_\alpha^{(k-1)}-\lambda_{\beta+1}^{(k)}}t^{\beta-\alpha})}
=\prod_{k=1}^m \psi_{\lambda^{(k)}/\lambda^{(k-1)}}(q,t),
\end{align*}
which is $\psi_T(q,t)$.
On the other hand if $T\in\Tbr{\lambda;m}\setminus\SSTb{\lambda;m}$,
one can see that $\eqref{eq:psi:5}=0$.
Using Proposition \ref{prp:tableau}, 
we have the desired equality.

\subsection{Proof of Proposition \ref{prp:tt*}}
\label{subsec:prf:tt*}

First we rewrite the relation of $\psi^{\pm}(z)$ and $x^{\pm}(w)$ 
given in Definition \ref{dfn:uqt} into the next adjoint form.
\begin{align*}
&\exp\Big( \sum_{n>0} \ad_{b_n}\gamma^{n/2}z^{-n} \Big) x^{\pm}(w)
=\exp\Big( \mp\sum_{n>0} 
           \dfrac{1}{n}(1-q^n)(1-t^{-n})(1-p^{-n})\gamma^{\mp n/2}
           \big(\dfrac{w}{z}\big)^n
     \Big) x^{\pm}(w),
\\
&\exp\Big(  -\sum_{n>0} \ad_{b_{-n}}\gamma^{n/2}z^{n} \Big) x^{\pm}(w)
=\exp\Big(\pm\sum_{n>0} 
          \dfrac{1}{n}(1-q^n)(1-t^{-n})(1-p^{-n})\gamma^{\mp n/2}
          \big(\dfrac{w}{z}\big)^n
      \Big) x^{\pm}(w).
\end{align*}
Here we used the exponential form \eqref{eq:psi_boson} of $\psi^{\pm}$.
Then we see that
\begin{align}
\nonumber
 \alpha(z) x^{\pm}(w) \alpha(z)^{-1}
&
=\exp\Big( -\sum_{n>0} \ad_{b_{-n}}\dfrac{z^{n}}{\gamma^{n}-\gamma^{-n}}
     \Big) x^{\pm}(w)
\\
\label{eq:alpha_x}
&
=\exp\Big( \pm\sum_{n>0} \dfrac{1}{n}
           \dfrac{(1-q^n)(1-t^{-n})(1-p^{-n})}{\gamma^{n}-\gamma^{-n}}
           \gamma^{-n/2\mp n/2} \big(\dfrac{z}{w}\big)^{n}
     \Big) x^{\pm}(w),
\\
\nonumber
 \beta(z) x^{\pm}(w) \beta(z)^{-1}
&
=\exp\Big( \sum_{n>0} \ad_{b_{n}}\dfrac{z^{-n}}{\gamma^{n}-\gamma^{-n}}
     \Big) x^{\pm}(w)
\\
\label{eq:beta_x}
&
=\exp\Big( \mp\sum_{n>0} \dfrac{1}{n}
           \dfrac{(1-q^n)(1-t^{-n})(1-p^{-n})}{\gamma^{n}-\gamma^{-n}}
           \gamma^{-n/2\mp n/2} \big(\dfrac{w}{z}\big)^{n}
     \Big) x^{\pm}(w).
\end{align}
We also prepare the operator product of $\alpha(w)$ and $\beta(z)$,
which is easily obtained from the definitions \eqref{eq:alpha_beta} 
and the commutation relations \eqref{eq:boson_b} of $b_n$'s:
\begin{align}
\label{eq:alpha_beta_ope}
 \beta(z)\alpha(w)
=\alpha(w)\beta(z)
 \exp\Big(\sum_{n>0}\dfrac{1}{n}
          \dfrac{(1-q^n)(1-t^{-n})(1-p^{-n})}{\gamma^n-\gamma^{-n}}
          \gamma^{-n}\big(\dfrac{w}{z}\big)^{n}
     \Big).
\end{align}

\subsubsection{Proof of (1)}
Using \eqref{eq:alpha_beta_ope} and \eqref{eq:alpha_x},
we see that
\begin{align*}
&\alpha(z) t(w) \alpha(z)^{-1}
=\alpha(z) \alpha(w) x^+(w) \beta(w) \alpha(z)^{-1}
\\
&
=\alpha(w) \alpha(z) x^+(w) \alpha(z)^{-1} \beta(w)
 \times 
 \exp\Big(-\sum_{n>0}\dfrac{1}{n}
          \dfrac{(1-q^n)(1-t^{-n})(1-p^{-n})}{\gamma^n-\gamma^{-n}}
          \gamma^{-n}\big(\dfrac{z}{w}\big)^{n}
     \Big)
\\
&=\alpha(w) x^+(w) \beta(w)
 \times 
 \exp\Big(-\sum_{n>0}\dfrac{1}{n}
           \dfrac{(1-q^n)(1-t^{-n})(1-p^{-n})}{\gamma^n-\gamma^{-n}}
           \gamma^{-n}\big(\dfrac{z}{w}\big)^{n}
\\
&\phantom{=\alpha(w) x^+(w) \beta(w)  \times  \exp\Big(|}
          +\sum_{n>0} \dfrac{1}{n}
           \dfrac{(1-q^n)(1-t^{-n})(1-p^{-n})}{\gamma^{n}-\gamma^{-n}}
           \gamma^{-n} \big(\dfrac{z}{w}\big)^{n}
     \Big)
=t(w).
\end{align*}
Thus we have $[t(z),\alpha(w)]=0$.
The other relations $[t(z),\beta(w)]=0$, 
$[t^*(z),\alpha(w)]=[t^*(z),\beta(w)]=0$,
$[t(z),\psi^{\pm}(w)]=[t^*(z),\psi^{\pm}(w)]=0$ 
also follow from equations \eqref{eq:alpha_x}-\eqref{eq:alpha_beta_ope} 
and we omit the detail.

\subsubsection{Proof of (2)}

Using the commutativity $[t(z),\alpha(w)]=0$ given in (1), we have
\begin{align*}
&A(w/z)t(z)t(w)
=A(w/z)\alpha(z)x^+(z)\beta(z)\alpha(w)x^+(w)\beta(w)
=\alpha(z)\alpha(w)x^+(z)x^+(w)\beta(z)\beta(w)
\\
&\times
\exp\Big(\sum_{n>0}\dfrac{1}{n}
         \dfrac{(1-q^n)(1-t^{-n})(1-p^{-n}\gamma^{-2n})}{1-\gamma^{-2n}}
         \big(\dfrac{w}{z}\big)^{n}
         -\sum_{n>0}\dfrac{1}{n}
          \dfrac{(1-q^n)(1-t^{-n})(1-p^{-n})}{\gamma^n-\gamma^{-n}}
          \gamma^{-n}\big(\dfrac{w}{z}\big)^{n}
    \Big).
\end{align*}
Here the first summation in the exponential comes from the $A(w/z)$,
and the second from transposition of $\beta(w) x^+(z)$ using \eqref{eq:beta_x}.
Thus we have
\begin{align*}
A(w/z)t(z)t(w)
&=\alpha(z)\alpha(w)x^+(z)x^+(w)\beta(z)\beta(w)
\times
\exp\Big(\sum_{n>0}\dfrac{1}{n} (1-q^n)(1-t^{-n}) \big(\dfrac{w}{z}\big)^{n}
    \Big)
\\
&=\dfrac{(1-q \tfrac{w}{z})(1-t^{-1}\tfrac{w}{z})}
        {(1-\tfrac{w}{z})(1-p \tfrac{w}{z})}
   \alpha(z)\alpha(w)x^+(z)x^+(w)\beta(z)\beta(w).
\end{align*}
Then
\begin{align}
\nonumber
&A(w/z)t(z)t(w)-A(z/w)t(w)t(z)
=\alpha(z)\alpha(w)
\\
\label{eq:tt:1}
&\phantom{=}\times
\Big[
 \dfrac{(1-q\tfrac{w}{z})(1-t^{-1}\tfrac{w}{z})}
       {(1-\tfrac{w}{z})(1-p \tfrac{w}{z})} x^+(z)x^+(w)
-\dfrac{(1-q\tfrac{z}{w})(1-t^{-1}\tfrac{z}{w})}
       {(1-\tfrac{z}{w})(1-p\tfrac{z}{w})} x^+(w)x^+(z)
\Big]
\\
\nonumber
&\phantom{=}\times
\beta(z)\beta(w).
\end{align}
Now recall the relation of $x^+(z)$ and $x^+(w)$ given 
in Definition \ref{dfn:uqt}:
\begin{align}\label{eq:Gxx}
-(\tfrac{z}{w})^3 G^+(\tfrac{z}{w})x^+(z)x^+(w)=G^+(\tfrac{z}{w})x^+(w)x^+(z).
\end{align}
Using this equation, the line \eqref{eq:tt:1} is rewritten into
\begin{align*}
\eqref{eq:tt:1}&=
\Big[
 \dfrac{1}{(1-\tfrac{w}{z})(1-p \tfrac{w}{z})(1-p^{-1} \tfrac{w}{z})}    
+\dfrac{(\tfrac{z}{w})^3}
       {(1-\tfrac{z}{w})(1-p\tfrac{z}{w})(1-p^{-1}\tfrac{z}{w})} 
\Big] G^+(\tfrac{w}{z})x^+(w)x^+(z)
\\
&=
\Big[
 \dfrac{\delta(\tfrac{w}{z})}{(1-p^{-1})(1-p)}    
+\dfrac{\delta(p\tfrac{w}{z})}{(1-p^{-1})(1-p^{-2})} 
+\dfrac{\delta(p^{-1}\tfrac{w}{z})}{(1-p)(1-p^2)} 
\Big] G^+(\tfrac{w}{z})x^+(w)x^+(z).
\end{align*}
Now from \eqref{eq:Gxx} and $G^+(1)\neq0$,
we see that $\delta(w/z)G^+(w/z)x^+(w)x^+(z)=0$.
We also find from \eqref{eq:Gxx} and  $G^+(p^{-1})=0$ that
$\delta(p \tfrac{w}{z})G^+(\tfrac{w}{z})x^+(w)x^+(z)
=\delta(p \tfrac{w}{z})G^+(p^{-1})x^+(p w)x^+(w)$.
Similarly from \eqref{eq:Gxx} and  $G^+(p)\neq0$ we have
$\delta(p^{-1} \tfrac{w}{z})G^+(\tfrac{w}{z})x^+(w)x^+(z)
=\delta(p^{-1} \tfrac{w}{z})G^+(p)x^+(p z)x^+(z)$.
Thus after a short calculation we have
\begin{align*}
\eqref{eq:tt:1}=
\dfrac{(1-t^{-1})(1-q)}{1-p}
\Big[
 \delta(p^{-1}\tfrac{w}{z}) x^+(p z)x^+(z)
-\delta(p     \tfrac{w}{z}) x^+(p w)x^+(w)
\Big].
\end{align*}
Then we have the desired consequence \eqref{eq:Att}.

The equation \eqref{eq:Att*} can be similarly shown, 
so that we omit the detail.

\subsubsection{Proof of (3)}
We apply the same method as in (2).
Recalling Remark \eqref{rmk:t*}, we calculate 
$B(\gamma p w/z)t(z)t^*(\gamma p w)
-B(\gamma^{-1}p^{-1}z/w)t^*(\gamma p w)t(z)$.
From the definition \eqref{eq:B} of $B(z)$, 
the commutativity $[t(z),\alpha(w)]=0$ given in (1) 
and the formula \eqref{eq:alpha_beta_ope}, we have
\begin{align*}
&B(\gamma p w/z)t(z)t^*(\gamma p w)
=B(\gamma p \tfrac{w}{z})
\alpha(z)x^+(z)\beta(z)
\alpha(\gamma w)^{-1}x^-(w)\beta(\gamma^{-1}w)^{-1}
\\
&=
\alpha(z)\alpha(\gamma w)^{-1} x^+(z) x^-(w) \beta(z)\beta(\gamma^{-1}w)^{-1}
\\
&\phantom{=}
\times
 \exp\Big(\sum_{n>0}\dfrac{1}{n}
          \dfrac{(1-q^n)(1-t^{-n})(1-p^{-n})}{\gamma^n-\gamma^{-n}}
          \big(\dfrac{w}{z}\big)^{n}
\\
&\phantom{= \times  \exp\Big(}
         +\sum_{n>0}\dfrac{1}{n}
          \dfrac{(1-q^n)(1-t^{-n})(\gamma^{-2n}p^{-2n}-\gamma^{-2n}p^{-n})}
                {1-\gamma^{-2n}}
          \gamma^n p^{-n} \big(\dfrac{w}{z}\big)^{n}\Big)
\\
&=
\alpha(z)\alpha(\gamma w)^{-1} x^+(z) x^-(w) \beta(z)\beta(\gamma^{-1}w)^{-1}.
\end{align*}
A similar calculation shows that
$B(\gamma p \tfrac{z}{w})t^*(\gamma p w) t(z)=
\alpha(z)\alpha(\gamma w)^{-1} x^-(w) x^+(z) \beta(z)\beta(\gamma^{-1}w)^{-1}
$.
Thus we have
\begin{align*}
& B(\gamma p \tfrac{w}{z})t(z)t^*(\gamma p w)
 -B(\gamma p \tfrac{z}{w})t^*(\gamma p w) t(z)
\\
&=
\alpha(z)\alpha(\gamma w)^{-1}
[ x^+(z) x^-(w)- x^-(w) x^+(z)]
\beta(z)\beta(\gamma^{-1}w)^{-1}
\end{align*}
Using the expression of $[x^+(z),x^-(w)]$ given in Definition \ref{dfn:uqt}, 
the expansion \eqref{eq:psi_boson} and the defnition \eqref{eq:alpha_beta}, 
one may immediately find that
\begin{align*}
  B(\gamma p \tfrac{w}{z})t(z)t^*(\gamma p w)
 -B(\gamma p \tfrac{z}{w})t^*(\gamma p w)t(z)
=\dfrac{(1-q)(1-t^{-1})}{1-p}
 \Big(\delta(\gamma^{-1}\tfrac{z}{w})\psi_0^+
-\delta(\gamma\tfrac{z}{w})\psi_0^-\Big).
\end{align*}
Replacing $w$ in the above equation with $\gamma^{-1}p^{-1} w$,
we have the desired equation \eqref{eq:Btt*}.

\subsection{Proof of Proposition \ref{prp:Lambda}}
\label{subsec:prf:Lambda}

Let us define
$a_{n,(i)}\seteq 
1\otimes \cdots \otimes 1\otimes a_{n}\otimes 1\otimes \cdots 1$,
where $a_{n}$ sits in the $i$-th tensor component.
Then from \eqref{eq:Delta_boson} and \eqref{eq:rhom} one finds that
$\rho_y^{(m)}(b_{n})
=-\sum_{i=1}^m a_{n,(i)}(1-t^n)(1-p^{-|n|}) p^{(m-i+1)|n|/2}/|n|$.
Thus we have
\begin{align}
\label{eq:rho_alpha}
&\rho_y^{(m)}(\alpha(z))=
\prod_{i=1}^m \alpha_{(i)}^m(z),
\quad
\alpha_{(i)}^m(z) \seteq
\exp\Big(\sum_{n>0}\dfrac{1}{n}\dfrac{p^{(m-i+1)n/2}(1-t^{-n})(1-p^{-n})}
                                     {p^{-m n/2}-p^{m n/2}}a_{-n,(i)}z^n\Big),
\\
\label{eq:rho_beta}
&
\rho_y^{(m)}(\beta(z))=
\prod_{i=1}^m \beta_{(i)}^m(z),
\quad
\beta_{(i)}^m(z)\seteq
\exp\Big(-\sum_{n>0}\dfrac{1}{n}\dfrac{p^{(m-i+1)n/2}(1-t^n)(1-p^{-n})}
                                      {p^{-m n/2}-p^{m n/2}}a_{n,(i)}z^{-n}\Big).
\end{align}


\subsubsection{Proof of (1)}
We calculate each tensor component of $\Lambda_i(z)\Lambda_j(w)$.
First assume $i=j$.

If $k>i$, then the $k$-th tensor component comes from 
$\alpha_{(k)}^m(z) \beta_{(k)}^m(z) \alpha_{(k)}^m(w)  \beta_{(k)}^m(z)$.
Under the normal ordering, the following coefficient arises.
\begin{align}
\label{eq:prf:L:1:i=j:k>i}
\exp\Big(-\sum_{n>0}\dfrac{1}{n}(1-q^n)(1-t^{-n})
         \big(\dfrac{1-p^{-n}}{1-p^{m n}}\big)^2 p^{(2m-k+1)n} 
         \big(\dfrac{w}{z}\big)^n\Big).
\end{align}
For $k=i$, the $i$-th tensor component comes from 
$\alpha_{(k)}^m(z)\eta(p^{-(i-1)/2}z) \beta_{(k)}^m(z) 
 \alpha_{(k)}^m(w)\eta(p^{-(i-1)/2}w) \beta_{(k)}^m(w)$.
Under the normal ordering, the following coefficient arises.
\begin{align}
\label{eq:prf:L:1:i=j:k=i}
\begin{split}
&\exp\Big(-\sum_{n>0}\dfrac{1}{n}(1-q^n)(1-t^{-n})\big(\dfrac{w}{z}\big)^n
\\
&\phantom{\exp\Big(-\sum_{n>0}}
         \big\{\big(\dfrac{1-p^{-n}}{1-p^{m n}}\big)^2 p^{(2m-i+1)n}
           +\dfrac{1-p^{-n}}{1-p^{m n}}p^{m n}
           +\dfrac{1-p^{-n}}{1-p^{m n}}p^{(m-i+1)n} 
           +1
         \big\}\Big).
\end{split}
\end{align}
If $k<i$, then the $k$-th tensor component is
$\alpha_{(k)}^m(z)\varphi^{-}(p^{-(2k-1)/4}z) \beta_{(k)}^m(z)
 \alpha_{(k)}^m(w)\varphi^{-}(p^{-(2k-1)/4}w)$ 
$\beta_{(k)}^m(w)$.
The normal ordering coefficient is
\begin{align}
\label{eq:prf:L:1:i=j:k<i}
&\exp\Big(-\sum_{n>0}\dfrac{1}{n}(1-q^n)(1-t^{-n})\big(\dfrac{w}{z}\big)^n
         \big\{\big(\dfrac{1-p^{-n}}{1-p^{m n}}\big)^2 p^{(2m-k+1)n}
              +\dfrac{(1-p^{-n})^2}{1-p^{m n}} p^{(m-k+1)n}
         \big\}\Big).
\end{align}
By simple calculations, 
the product of \eqref{eq:prf:L:1:i=j:k>i}, \eqref{eq:prf:L:1:i=j:k=i} and
\eqref{eq:prf:L:1:i=j:k<i} is shown to be $f_{1,m}(w/z)^{-1}$.
Thus the statement holds.

Next we consider the case $i<j$.
If $k<i$, then the normal order coefficient is the same as 
\eqref{eq:prf:L:1:i=j:k<i}.
For $k=i$, the normal order coefficient is 
\begin{align}
\label{eq:prf:L:1:i<j:k=i}
\begin{split}
&\exp\Big(-\sum_{n>0}\dfrac{1}{n}(1-q^n)(1-t^{-n})\big(\dfrac{w}{z}\big)^n
\\
&\phantom{\exp\Big(-\sum_{n>0}}
         \big\{\big(\dfrac{1-p^{-n}}{1-p^{m n}}\big)^2 p^{(2m-i+1)n}
           +\dfrac{1-p^{-n}}{1-p^{m n}}p^{m n}
           +\dfrac{(1-p^{-n})^2}{1-p^{m n}}p^{(m-i+1)n} 
           +1-p^{-n}
         \big\}\Big).
\end{split}
\end{align}
If $i<k<j$, then the normal order coefficient is
\begin{align}
\label{eq:prf:L:1:i<j:i<k<j}
&\exp\Big(-\sum_{n>0}\dfrac{1}{n}(1-q^n)(1-t^{-n})\big(\dfrac{w}{z}\big)^n
         \big\{\big(\dfrac{1-p^{-n}}{1-p^{m n}}\big)^2 p^{(2m-k+1)n}
           +\dfrac{(1-p^{-n})^2}{1-p^{m n}}p^{(m-k+1)n} 
         \big\}\Big).
\end{align}
If $k=j$, then the normal order coefficient is
\begin{align}
\label{eq:prf:L:1:i<j:k=j}
&\exp\Big(-\sum_{n>0}\dfrac{1}{n}(1-q^n)(1-t^{-n})\big(\dfrac{w}{z}\big)^n
         \big\{\big(\dfrac{1-p^{-n}}{1-p^{m n}}\big)^2 p^{(2m-j+1)n}
           +\dfrac{1-p^{-n}}{1-p^{m n}}p^{(m-j+1)n} 
         \big\}\Big).
\end{align}
If  $k>j$, then the normal order coefficient is 
\eqref{eq:prf:L:1:i=j:k>i}.
The product of \eqref{eq:prf:L:1:i=j:k<i}, \eqref{eq:prf:L:1:i<j:k=i},
\eqref{eq:prf:L:1:i<j:i<k<j}, \eqref{eq:prf:L:1:i<j:k=j}, 
\eqref{eq:prf:L:1:i=j:k>i} is equal to $f_{1,m}(w/z)^{-1}\gamma_{+}(z,w;q,p)$.
Thus we obtain the result.

The case $i>j$ is similar, so we omit the detail.

\subsubsection{Proof of (2)}
The desired equation is equivalent to
\begin{align*}
\rho_{y}^{(m)}(\alpha(z)\cdots\alpha(p^{m-1}z))
:\prod_{k=1}^{m}\widetilde{\Lambda}_k(p^{k-1}z):
\rho_{y}^{(m)}(\beta(z)\cdots\beta(p^{m-1}z))=1.
\end{align*}
We will show this equation by comparing each tensor component.

By \eqref{eq:rho_alpha}, the $k$-th tensor component of 
$\rho_{y}^{(m)}(\alpha(z)\cdots\alpha(p^{m-2}z))$ is equal to
\begin{align}
\exp\Big(-\sum_{n>0}\dfrac{1}{n}(1-t^{-n})p^{(2m-k-1)n/2}a_{-n}z^n\Big).
\label{eq:LL:2:k:alpha}
\end{align}
Similarly, the $k$-th tensor component of 
$\rho_{y}^{(m)}(\beta(z)\cdots\beta(p^{m-1}z))$ is equal to
\begin{align}
\exp\Big(\sum_{n>0}\dfrac{1}{n}(1-t^{-n})p^{(-k+1)n/2}a_{n}z^{-n}\Big).
\label{eq:LL:2:k:beta}
\end{align}
The $k$-th tensor component of 
$:\prod_{k=1}^{m}\widetilde{\Lambda}_k(p^{k-1}z):$ is 
\begin{align}
\nonumber
&:\eta(p^{-(k-1)/2}p^{k-1}z)
  \varphi^-(p^{-(2k-1)/4}p^{k}z)
  \varphi^-(p^{-(2k-1)/4}p^{k+1}z)\cdots
  \varphi^-(p^{-(2k-1)/4}p^{m-1}z):
\\
\label{eq:LL:2:k:L}
&
=\exp\Big( \sum_{n>0} \dfrac{1-t^{-n}}{n}p^{n(2m-k-1)/2}a_{-n}z^{n} \Big)
 \exp\Big(-\sum_{n>0} \dfrac{1-t^{n}}{n} p^{-n(k-1)/2}a_{n}  z^{-n} \Big)
\end{align}
It is easy to see that \eqref{eq:LL:2:k:alpha},\eqref{eq:LL:2:k:beta} 
and \eqref{eq:LL:2:k:L} cancel. 
Thus we have the consequnce.

\subsubsection{Proof of (3)}
The desired equation is equivalent to
\begin{align}
\begin{split}
\label{eq:LL:3}
&\rho_{y}^{(m)}(\alpha(p^{-1}z)\cdots\alpha(p^{m-2}z))
:\prod_{k=1}^{i-1}\widetilde{\Lambda}_k(p^{k-1}z)
 \prod_{l=i+1}^{m}\widetilde{\Lambda}_l(p^{l-2}z):
\rho_{y}^{(m)}(\beta(z)\cdots\beta(p^{m-1}z))
\\
&=\widetilde{\Lambda}_i^*(p^{(m-2)/2}z).
\end{split}
\end{align}
We will show this equation by comparing each tensor component.

As in \eqref{eq:LL:2:k:alpha} , the $k$-th tensor component of 
$\rho_{y}^{(m)}(\alpha(p^{-1}z)\cdots\alpha(p^{m-2}z))$ is equal to
\begin{align}
\exp\Big(-\sum_{n>0}\dfrac{1}{n}(1-t^n)p^{(2m-k-3)n/2}a_{-n}z^n\Big).
\label{eq:LL:3:k:alpha}
\end{align}
The $k$-th tensor component of 
$\rho_{y}^{(m)}(\beta(z)\cdots\beta(p^{m-1}z))$ is given by 
\eqref{eq:LL:2:k:beta}.

The $k$-th tensor component of 
$:\prod_{k=1}^{i-1}\widetilde{\Lambda}_k(p^{k-1}z)
 \prod_{l=i+1}^{m}\widetilde{\Lambda}_l(p^{l-2}z):$
depends on $k$.
If $k=i$, then by Lemma \ref{lem:tildeLambda} and some simple calculations,
the component turns out to be 
\begin{align}
\nonumber
&\varphi^-(p^{-(2i-1)/4}p^{i-1}z)\varphi^-(p^{-(2i-1)/4}p^{i}z)\cdots
 \varphi^-(p^{-(2i-1)/4}p^{m-2}z)\\
\label{eq:LL:3:L:i}
&=\exp\Big(-\sum_{n>0}\dfrac{1-t^{-n}}{n}p^{(i-3)/2}(1-p^{n(m-i)})
          a_{-n}z^n\Big).
\end{align}
Similarly, if $k<i$, then by Lemma \ref{lem:tildeLambda} the component is 
\begin{align}
\nonumber
&:\eta(p^{-(k-1)/2}p^{k-1}z)
  \varphi^-(p^{-(2k-1)/4}p^{k}z)
  \varphi^-(p^{-(2k-1)/4}p^{k+1}z)\cdots
  \varphi^-(p^{-(2k-1)/4}p^{m-2}z):
\\
\label{eq:LL:3:L:<i}
&
=\exp\Big( \sum_{n>0} \dfrac{1-t^{-n}}{n}p^{n(2m-k-3)/2}a_{-n}z^{n} \Big)
 \exp\Big(-\sum_{n>0} \dfrac{1-t^{n}}{n} p^{-n(k-1)/2}a_{n}  z^{-n} \Big):
\end{align}
If $k>i$, then by Lemma \ref{lem:tildeLambda} the component is 
\begin{align}
\nonumber
&:\eta(p^{-(k-1)/2}p^{k-2}z)
  \varphi^-(p^{-(2k-1)/4}p^{k-1}z)
  \varphi^-(p^{-(2k-1)/4}p^{k}z)\cdots
  \varphi^-(p^{-(2k-1)/4}p^{m-2}z):
\\
\label{eq:LL:3:L:>i}
&
=\exp\Big( \sum_{n>0} \dfrac{1-t^{-n}}{n}p^{n(2m-k-3)/2}a_{-n}z^{n} \Big)
 \exp\Big(-\sum_{n>0} \dfrac{1-t^{n}}{n} p^{-n(k-3)/2}a_{n}  z^{-n} \Big):
\end{align}

Then the $i$-th tensor component of \eqref{eq:LL:3} is the product of 
\eqref{eq:LL:3:k:alpha},\eqref{eq:LL:2:k:beta} and \eqref{eq:LL:3:L:i}.
After a short calculation, one finds that it is $\xi(p^{(i-2)/2}z)$,
which is the $i$-th component of $\widetilde{\Lambda}_i^*(p^{(m-2)/2}z)$.

If $k<i$, then the $k$-th tensor component of \eqref{eq:LL:3} is the product 
of \eqref{eq:LL:3:k:alpha},\eqref{eq:LL:2:k:beta} and \eqref{eq:LL:3:L:<i}.
It is $1$, that is, 
the $k$-th component of $\widetilde{\Lambda}_i^*(p^{(m-2)/2}z)$.

Finally, for $k>i$, the $k$-th tensor component of \eqref{eq:LL:3} is 
the product of \eqref{eq:LL:3:k:alpha},\eqref{eq:LL:2:k:beta} 
and \eqref{eq:LL:3:L:>i}.
It turns out to be $\varphi^-(p^{(2j-5)/4}z)$,
which is the $k$-th component of $\widetilde{\Lambda}_i^*(p^{(m-2)/2}z)$.

\subsubsection{Proof of (4)}
From the known identities \eqref{eq:prp:Lambda:1} and \eqref{eq:prp:Lambda:3},
it is not difficult to calculate 
$\Big[\prod_{k,l=1}^{m-1}f_{1.m}(p^{-k+l}w/z)\Big] 
 \Lambda_i^*(z)\Lambda_j^*(w)$ in terms of $\Lambda_k$'s.

First we consider the case $i=j$.
From the operator product \eqref{eq:prp:Lambda:1}, we have
\begin{align*}
&\Big[\prod_{k,l=1}^{m-1}f_{1,m}(p^{-k+l}\tfrac{w}{z})\Big] 
 \Lambda_i^*(z)\Lambda_i^*(w)
\\
&=\Big[\prod_{k=1}^{m-2}\prod_{l=k+1}^{m-1} \gamma_+(p^{-k+l}\tfrac{w}{z})\Big]
  \Big[\prod_{k=2}^{m-1}\prod_{l=1}^{k-1}   \gamma_-(p^{-k+l}\tfrac{w}{z})\Big]
  :\Lambda_i^*(z)\Lambda_i^*(w):
\\
&=\exp\Big(\sum_{n>0}\dfrac{1}{n}(1-q^n)(1-t^{-n})
           \dfrac{1-p^{-n(m-2)}}{1-p^{-n}} \dfrac{1-p^{n(m-1)}}{1-p^{n}} 
           \big(\dfrac{w}{z}\big)^n \Big)
:\Lambda_i^*(z)\Lambda_i^*(w):.
\end{align*}
Here we used the abbreviation 
$\gamma_\pm(\tfrac{w}{z})\seteq\gamma_\pm(z,w;q,p)$.
Then we have
\begin{align*}
&\Big[\prod_{k,l=1}^{m-1}f_{1,m}(p^{-k+l}\tfrac{w}{z})\Big] 
\times
\exp\Big(-\sum_{n>0}
           \dfrac{1}{n}(1-q^n)(1-t^{-n})
           \dfrac{1-p^{-n(m-2)}}{1-p^{-n}}
           \dfrac{1-p^{n(m-1)}}{1-p^{n}} 
           \big(\dfrac{w}{z}\big)^n \Big)
\\
&=\exp\Big(-\sum_{n>0}
           \dfrac{1}{n}
           \dfrac{(1-q^n)(1-t^{-n})(1-p^{(m-1)n}}{1-p^{m n}}
           \big(\dfrac{w}{z}\big)^n \Big)
=f_{1,m}(\tfrac{w}{z}).
\end{align*}
Thus the desired equation 
$f_{1,m}(\tfrac{w}{z})\Lambda_i^*(z)\Lambda_i^*(w)
=:\Lambda_i^*(z)\Lambda_i^*(w):$
is proved.

Next, note that the calculation of the case $i\neq j$ reduces to that of $k=i$.
If $i<j$, then
\begin{align*}
f_{1,m}(\tfrac{w}{z})\Lambda_i^*(z)\Lambda_j^*(w)
=:\Lambda_i^*(z)\Lambda_j^*(w): 
 \dfrac{\gamma_-(\tfrac{w}{z})^{j-i}}{\gamma_+(p\tfrac{w}{z})^{j-i-1}}
=:\Lambda_i^*(z)\Lambda_j^*(w): \gamma_-(\tfrac{w}{z}).
\end{align*}
At the last line we used the formula $\gamma_-(z)/\gamma_+(p z)=1$.
For the final case $i>j$, we have
\begin{align*}
f_{1,m}(\tfrac{w}{z})\Lambda_i^*(z)\Lambda_j^*(w)
=:\Lambda_i^*(z)\Lambda_j^*(w): 
 \dfrac{\gamma_+(\tfrac{w}{z})^{i-j}}{\gamma_-(p^{-1}\tfrac{w}{z})^{i-j-1}}
=:\Lambda_i^*(z)\Lambda_j^*(w): \gamma_+(\tfrac{w}{z}).
\end{align*}
Thus all the cases are proved.

\subsubsection{Proof of (5)}
This is similary shown as (2) and (3), so we omit the detail.

\begin{ack}
S.Y. is supported by JSPS Fellowships for Young Scientists (No.21-2241).
\end{ack}


\end{document}